\documentclass[smallextended,envcountsect,envcountsame]{svjour3}
\usepackage{amsfonts, amsmath, amssymb, url, dsfont}
\usepackage{graphicx}
\usepackage[english] {babel}
\usepackage{enumitem}
\usepackage{array}
\usepackage{cellspace}
\usepackage[ps, all,cmtip, dvips]{xy}

\DeclareMathOperator{\spa}{span}

\DeclareMathOperator{\II}{\textit{II}}

\newcommand{\R}{\mathbb{R}}

\newcommand{\Z}{\mathbb{Z}}
\newcommand{\La}{\mathbb{L}}

\newcommand{\eps}{\varepsilon}

\DeclareMathOperator{\su}{sup}

\DeclareMathOperator{\indre}{int}

\DeclareMathOperator{\exo}{exo}
\DeclareMathOperator{\tr}{Tr}
\newcommand{\Ha}{\mathcal{H}}

\smartqed
\title{Local digital algorithms for estimating the integrated mean curvature of $r$-regular sets}
\author{Anne Marie Svane}
\institute{Anne Marie Svane \at Department of Mathematics, Aarhus University, 8000 Aarhus C, Denmark,\\ \email{amsvane@imf.ua.dk}}
\titlerunning{Local digital algorithms for estimating the integrated mean curvature}
\begin{document}
\maketitle
\begin{abstract}
Consider the design based situation where an $r$-regular set is sampled on a random lattice. A fast algorithm for estimating the integrated mean curvature based on this observation is to use a weighted sum of $2\times \dotsm \times 2$ configuration counts. We show that for a randomly translated lattice, no asymptotically unbiased estimator of this type exists in dimension greater than or equal to three, while for stationary isotropic lattices, asymptotically unbiased estimators are plenty. Both results follow from a general formula that we state and prove, describing the asymptotic behavior of hit-or-miss transforms of $r$-regular sets.
\keywords{Binary image \and Design based set-up \and Configurations \and Mean curvature \and $r$-regular sets \and Hit-or-miss transform}
\subclass{94A08 \and 28A75 \and 60D05}
\end{abstract}
\section{Introduction}
Suppose we are given a digital image of some geometric object. 
In many practical situations within science, one is mainly interested in certain geometrical characteristics of the underlying object. These are the so-called intrinsic volumes $V_i$ and include the volume $V_d$, the surface area $2V_{d-1}$, the integrated mean curvature $2\pi(d-1)^{-1}V_{d-2}$, and the Euler characteristic $V_0$. Therefore, a time consuming reconstruction of the object is not of interest. 
Instead, we consider an algorithm for estimating the intrinsic volumes based only on local information.

We model a digital image of a compact set $X\subseteq \R^d$ as a binary image, i.e.\ as the set $X\cap \La$ where $\La\subseteq \R^d$ is some lattice.  The vertices of each $2\times \dotsm \times 2$ cell in the lattice may belong to either $X$ or $\R^d\backslash X$, yielding $2^{2^d}$ possible configurations. We then estimate $V_i$ as a weighted sum of the number of occurences of each configuration. The weights are functions of the lattice distance and we assume that they are homogeneous of degree $i$. The advantage of such local algorithms is that they are very efficiently implemented based on linearly filtering the image, see \cite{OM} for more on the computational aspects.

%

We apply these algorithms to the design based setting in which we sample a fixed compact set with a lattice that has been ramdomly translated. Ideally, the estimator should be unbiased, at least aymptotically when the resolution goes to infinity.  

Local estimators for $V_{d-1}$ have already been widely studied. In \cite{rataj}, Kiderlen and Rataj prove a formula for the asymptotic behavior of such an estimator. This was later applied by Ziegel and Kiderlen in \cite{johanna} to show that no asymptotically unbiased estimator for the surface area of the type described above can exist in dimension $d=3$. 

In this paper, we focus on the estimation of $V_{d-2}$. For $d=2$, $V_{d-2}$ is the Euler characteristic. It is well-known that estimating $V_0$ is impossible even in the simple case where $X$ is a polygon. More generally, Kampf has shown in \cite{kampf2} that no asymptotically unbiased estimator for $V_{d-2}$ exists on the class of finite unions of polytopes. In contrast, it was shown already in 1982 by Pavlidis in \cite{pavlidis} that unbiased estimators for $V_0$ do exist on a class of sets with sufficiently `smooth' boundary, namely the class of so-called $r$-regular sets. For this reason, we will require throughout the paper that $X$ is $r$-regular when we consider estimators for $V_{d-2}$ in higher dimensions. 

We are going to prove an extension to second order of Kiderlen and Rataj's asymptotic result \cite[Theorem 1]{rataj}.  In particular, we obtain a formula for the asymptotic mean of a local estimator for $V_{d-2} $. This was  done in \cite{am} for $d=2$ under somewhat stronger conditions. The formula allows us to deduce the following main theorem:
\begin{theorem}\label{main}
In dimension $d>2$, no weighted sum of $2\times \dotsm \times 2$ with homogeneous weights configuration counts defines an asymptotically unbiased estimator for $V_{d-2}$ on the class of $r$-regular sets.
\end{theorem}

This is contrary to the $d=2$ case, but it generalizes Kampf's result to the class of $r$-regular sets. It is proved as Theorem \ref{main'} below. The counterexamples can be chosen very simply to be of the form $P\oplus B(r)$ where $B(r)$ is the ball of radius $r$ and $P=\bigoplus_{i=1}^k [0,u_i]$ where $u_1,\dots , u_k\in \R^d$ are orthonormal vectors and $\oplus $ is the Minkowski sum.

We give a formal definition of the type of local algorithm we consider in Section~\ref{notation}, and in Section~\ref{design} we explain the design based setting and recall some known results. In Section~\ref{regular} and~\ref{rotinv}, we prove some general results on hit-or-miss transforms of $r$-regular sets with finite structuring elements. As a corollary, we obtain formulas for the asymptotic behavior of the mean estimator for $V_{d-2}$ in Section~\ref{apply}. In Section~\ref{3Diso}, we apply this to find all asymptotically unbiased estimators in 3D under the assumption that the lattice $\La$ is isotropic. In the remaining two sections, we investigate the case where the lattice is not assumed to be isotropic. In Section~\ref{EC}, we recover the Pavlidis' result that an asymptotically unbiased estimator for $V_0$ does exist in dimension $d=2$. Finally, we prove Theorem~\ref{main} in Section~\ref{higher}.

\section{Local estimators for intrinsic volumes}\label{notation}
Let $C$ denote the unit square $[0,1]^d$ in $\R^d$ and let $C_0$ be the set of vertices in $C$. The vectors of the standard basis in $\R^d$ will be denoted by $e_1,\dots ,e_d$. We enumerate the elements of $C_0$ as follows: for $x\in C_0$ we write $x=x_i$ where
\begin{equation*}
i=\sum_{k=1}^d 2^{k-1}\mathds{1}_{\langle x,e_k\rangle=1 }.
\end{equation*}
Here $\mathds{1}_{\langle x,e_k\rangle=1}$ is the indicator function.
A $2\times \dots \times 2$ configuration is a subset $\xi \subseteq C_0$. There are $2^{2^d}$ possible configurations. We denote these by $\xi_l$ for $l=0,\dots,2^{2^d}-1$ where the configuration $\xi$ is assigned the index 
\begin{equation*}
l=\sum_{i=0}^{2^{2^d}-1} 2^i \mathds{1}_{x_i\in \xi}.
\end{equation*}
One could of course consider estimators based on $n\times \dotsm \times n$ configurations as well. The formulas we obtain in Section \ref{regular} and \ref{rotinv} apply to this case as well, but we treat only estimators based on $2\times \dotsm \times 2$ configurations in this paper. 
 
Let $\Z^d$ denote the standard lattice in $\R^d$.
More generally, we shall consider orthogonal lattices $a\La(c,R)=aR(\Z^d + c)$ where $c\in C$ is a translation vector, $R\in SO(d)$ is a rotation, and $a>0$ is the lattice distance. Then $C(a\La)$, $C_0(a\La)$, and $\xi_l(a\La)$ will denote the corresponding transformations of $C$, $C_0$, and $\xi_l$, respectively. We leave the lattice out of the notation whenever it is clear from the context. The generalization to the case where $\La$ is a general linear transformation of $\Z^d$ is straightforward and is left to the reader.

The elements of $\xi_l$ are referred to as the `foreground' or `black' pixels and will also sometimes be denoted by $B_l$, while the vertices of the complement $W_l=C_0\backslash \xi_l =\xi_{2^{2^d}-l}$ are referred to as the `background' or `white' pixels.

Now let $X\subseteq \R^d$ be a compact set observed on the lattice $a\La $. Based on the set $X\cap a\La$ we want to estimate the intrinsic volumes $V_i(X)$ for $i=0,\dots ,d$. For a general definition of $V_i$ in the case where $X$ is polyconvex, see \cite{schneider}. In this paper, we will only need the $V_i$ introduced at the beginning of the introduction. In order for $V_i$ to be well-defined and for $X\cap a\La$ to contain enough information about $X$, we will need some regularity conditions on $X$. These will be specified later. 

Our approach is to consider a local algorithm based on the observations of $X$ on the $2\times  \dotsm \times 2$  cells $C_z$ of $a\La$, where $C_z=z+C(a\La)$ for $z\in a\La(0,R)$.
The number of occurences of the configuration $\xi_l$ is  
\begin{equation*}\label{indic}
N_l(X\cap a\La) = \sum_{z \in a\La(0,R) } \mathds{1}_{X \cap (z+C_0(a\La)) = z+\xi_l(a\La)}.
\end{equation*}
Note that $N_l$ depends only on $X\cap a\La$, as 
\begin{equation*}
X \cap (z+C_0(a\La))=(X\cap a\La) \cap (z+C_0(a\La)).
\end{equation*}

If $\Phi_i(X;\cdot)$ denotes the $i$th curvature measure, normalized as in \cite{schneider},
\begin{equation*}
V_i(X)=\Phi_i(X;\R^d) = \sum_{z\in a\La(0,R)} \Phi_i(X;C_{z}^0)
\end{equation*}
where 
\begin{equation*}
C_{z}^0=z + Ra([0,1)^d+c).
\end{equation*}

We estimate each term in the sum based on the only information available about $ X \cap C_z$, namely the set $X\cap (z+C_0(a\La))$. If $X\cap (z+C_0(a\La))=z+\xi_l(a\La)$, we estimate $\Phi_i(X;C_{z}^0)$ by some $w_l^{(i)}(a)\in \R$, leading to an estimator of the form
\begin{equation}\label{estdef}
\hat{V}_i(X\cap a\La)=\sum_{l=0}^{2^{2^d}-1}  w_l^{(i)}(a){N}_l(X\cap a\La).
\end{equation}
The $w_l^{(i)}(a)$ are referred to as the weights.

Let $\mathcal{M}$ be the set of rigid motions and reflections preserving $C_0$. If $|\mathcal{M}|$ is the cardinality of $\mathcal{M}$,
\begin{equation*}
\hat{V}_i'(X\cap a\La)=\frac{1}{|\mathcal{M}|}\sum_{M\in \mathcal{M}} \hat{V}_i(M(X\cap a\La)).
\end{equation*}
is another estimator of the form \eqref{estdef} and the bias of $\hat{V}_i'(X)$ is the average of the biases of $\hat{V}_i$ on the sets $MX$, since $V_i(X)$ is motion and reflection invariant. Hence the worst possible bias of $\hat{V}_i'$ on the sets $MX$ is smaller than that of $\hat{V}_i$. Thus, in the search for unbiased estimators, it is enough to consider estimators with weights satisfying $w_{l_1}^{(i)}(a)=w_{l_2}^{(i)}(a)$ whenever $\xi_{l_1} = M\xi_{l_2}$ for some $M\in \mathcal{M}$. 

As $V_i$ is homogeneous of degree $i$, i.e.\ $V_i(aX)=a^iV_i(X)$, we will require the estimator to satisfy
\begin{equation*}
\hat{V}_i(aX\cap a\La)=a^i V_i(X\cap \La),
\end{equation*}
corresponding to weights of the form $w_l^{(i)}(a)=a^iw_l^{(i)}$ where $w_l^{(i)}\in \R$ are constants.

If $\eta_j^d$, $j\in J$, denote the equivalence classes of configurations under the action of $\mathcal{M}$, we end up with an estimator of the form
\begin{equation}\label{motion}
\hat{V}_i(X\cap a\La)=a^{i}\sum_{j\in J} w_j^{(i)}\bar{N}_j(X\cap a\La)
\end{equation}
where $w_j^{(i)}\in \R$ and 
\begin{equation*}
\bar{N}_j=\sum_{l:\xi_l\in \eta_j^d} N_l.
\end{equation*}

\section{The design based setting}\label{design}
In the design based setting we observe a fixed set $X\subseteq \R^d$ on a random lattice.  If the lattice is of the form $a\La=aR(\Z^d+c)$ where $c\in C$ and $R\in SO(d)$ are both uniform random and mutually independent, we shall speak of a stationary isotropic lattice. If $a\La = a(R\Z^d+c)$ where the translation vector $c\in C$ is uniform random while $R\in SO(d)$ is now fixed, we refer to it as a stationary non-isotropic lattice. In both cases, the local estimator \eqref{motion} is now a random variable with mean 
\begin{equation*}
E\hat{V}_i(X\cap a\La)= a^i\sum_{j\in J} w_j^{(i)}E\bar{N}_j(X\cap a\La).
\end{equation*}

Ideally, this would equal $V_i(X)$. However, this is generally not true in finite resolution, i.e.\ for $a>0$. Instead, we consider the asymptotic behavior of $E\hat{V}_i(X)$ as $a$ tends to 0. This is obtained by explicit formulas for the asymptotic behavior of $a^{i}EN_l$ when $a\to 0$.

Since $N_0$ is infinite, $w_0^{(i)}$ must equal zero in order for $\hat{V}_i$ to be well-defined. All other $EN_l$ are of order $O(a^{1-d})$, see \eqref{V1} below, except $EN_{2^{2^d}-1}$. In fact, for all the sets $X$ we shall consider, 
\begin{equation*}
\lim_{a\to 0}a^{d-i}E\hat{V}_i(X) = w_{2^{2^d}-1}^{(i)}V_d(X),
\end{equation*}
see e.g.\ \cite{OM}. Thus for $i<d$, we must require $w_{2^{2^d}-1}^{(i)}=0$, otherwise the limit $\lim_{a\to 0} E \hat{V}_i(X\cap a\La)$ does not exist. 

For the surface area, it was shown by Kiderlen and Rataj \cite[Theorem~5]{rataj} that if $X$ is a full-dimensional  compact gentle set and $\La$ is a stationary non-isotropic lattice,
\begin{equation}\label{V1}
\lim_{a\to 0} a^{d-1}EN_l(X\cap a\La) = \int_{\partial X }(-h(B_l\oplus \check{W}_l,n))^+d\Ha^{d-1}
\end{equation}
where for a set $S\subseteq \R^d$, $h(S,n)=\su \{\langle s,n\rangle\mid s\in S\}$ for $n\in S^{d-1}$ is the support function, $\check{S}=\{-s\mid s \in S\}$, and $\oplus$ is the Minkowski sum. Moreover, $x^+=\textrm{max}\{0,x\}$ for $x\in \R$,   and $\Ha^{k}$ denotes the $k$th Hausdorff measure.
The notion of a gentle set is explained in \cite{rataj}.

This result was later used by Ziegel and Kiderlen in \cite{johanna} to prove that there does not exist an asymptotically unbiased local estimator for the surface area of polygons in dimension $d=3$. 

Actually, Kiderlen and Rataj proved a much more general theorem, namely \cite[Theorem 1]{rataj}. We shall state the theorem here in a special case for later comparison:
\begin{theorem}[Kiderlen, Rataj]\label{thmV1}
Let $X\subseteq \R^d$ be a closed gentle set, $A\subseteq \R^d$ a bounded Borel set, and $B,W\subseteq \R^d$ two non-empty finite sets. Then
\begin{align}
\lim_{a\to 0} & a^{-1}\Ha^d(\xi_{\partial X}^{-1}(A)\cap (X\ominus a\check{B})\backslash(X\oplus a\check{W}))=\int_{\partial X\cap A}(-h(B\oplus \check{W},n))^+d\Ha^{d-1}\nonumber
\\&=\int_{\partial X\cap A}((-h(B,n))-h(\check{W},n) )\delta_{(B,W)}(n)d\Ha^{d-1}.\label{HaV1}
\end{align}
\end{theorem}
Here $\ominus $ is the Minkowski set difference. The set
\begin{equation*}
(X\ominus a\check{B})\backslash(X\oplus a\check{W})=  \{z\in \R^d \mid z+aB\subseteq X,z+aW\subseteq \R^d\backslash X\}
\end{equation*}
is called the hit-or-miss transform of $X$.
If $\exo(\partial X)$ denotes the set of points in $\R^d$ that do not have a unique closest point in $\partial X$, then $\xi_{\partial X}$ is the function $\xi_{\partial X}:\R^d\backslash \exo(\partial X) \to  \partial X$ that takes a point $z$ to the point in $\partial X$ closest to $z$. In the last line, the integral has just been rewritten in a form similar to what we shall later obtain with the notation
\begin{equation*}
\delta_{(B,W)}(n)=\mathds{1}_{\{h(B\oplus \check{W},n)<0\}}.
\end{equation*}

Equation \eqref{V1} follows from Theorem \ref{thmV1} and the observation that
\begin{align}\nonumber
a^dEN_l &= a^d\int_{ C}\sum_{z \in a\La(0,R) } \mathds{1}_{\{X \cap (z+C_0(a\La(c,R))) = z+ \xi_l(a\La(c,R))\}}dc\\ \label{HEN}
&=\Ha^d( z\in \R^d \mid z+aB_l\subseteq X,z+aW_l\subseteq \R^d\backslash X)\\
&=\Ha^d( (X\ominus a\check{B}_l)\backslash(X\oplus a\check{W}_l)).\nonumber
\end{align}

In the following section, we will consider the second order asymptotic behavior of 
\begin{equation*}
\Ha^d(\xi_{\partial X}^{-1}(A)\cap (X\ominus a\check{B})\backslash(X\oplus a\check{W}))
\end{equation*}
for $r$-regular sets $X$ when $a$ tends to zero. The main result is a formula similar to \eqref{HaV1} but with the support functions replaced by certain quadratic terms. 
Choosing $(B,W)=(B_l,W_l)$, Equation \eqref{HEN} shows that this has implications for the asymptotic behavior of $a^{d-2}EN_l$ and thus for the asymptotic mean of $\hat{V}_{d-2}$.  

\section{Hit-or-miss transforms of $r$-regular sets}\label{regular}
As explained in the introduction, estimating $V_i$ causes problems for $i<d-1$ even for polygons, so we need some strong assumptions on $X$. Thus we consider the class of so-called $r$-regular sets:

\begin{definition}\label{defreg}
A closed subset $X\subseteq \R^d $ is called $r$-regular for $r>0$ if for all $x\in \partial X$ there exists two balls $B_i$ and $B_o$ of radius $r$ both containing $x$ such that $B_i\subseteq X$ and $\indre (B_o)\subseteq \R^d\backslash X$. 
\end{definition}
The definition implies that $\partial X$ is a $C^1$ manifold, see e.g.\ \cite{federer}, and to all $x\in \partial X$ there is a unique outward pointing normal vector $n(x)$. Federer showed in \cite{federer} that the normal vector field $n$ is $\Ha^{d-1}$-almost everywhere differentiable. In particular, its principal curvatures $k_i$ can be defined almost everywhere  as the eigenvalues of the differential $dn$ corresponding to the orthogonal principal directions $e_i$. This generalizes the definition for $C^2$ manifolds. Note for later that each $k_i$ is bounded by ${r}^{-1}$.

Federer uses the principal curvatures to generalize the curvature measures for convex sets, see e.g. \cite{schneider}, to the much larger class of sets of positive reach which includes the class of $r$-regular sets.
In particular, $2\pi(d-1)^{-1} V_{d-2}$ is defined as the integrated mean curvature, i.e.
\begin{equation*}
V_{d-2}(X)=\frac{1}{2\pi} \int_{\partial X} (k_1+\dotsm+k_{d-1}) d\Ha^{d-1}.
\end{equation*}

The notion of principal curvatures also allows for a definition of the second fundamental form $\II_x$ on the tangent space $T_x\partial X$ for $\Ha^{d-1}$-almost all $x\in \partial X$, similar to the definition for $C^2$ manifolds. For $\sum_{i=1}^{d-1}\alpha_ie_i\in T_x\partial X$, $\II_x$ is given by
\begin{equation*}
\II_x\left(\sum_{i=1}^{d-1}\alpha_ie_i \right)= \sum_{i=1}^{d-1}k_i(x)\alpha_i^2
\end{equation*}
whenever $d_xn$ is defined.
Note that $\tr ( \II)=k_1+\dotsm+k_{d-1}$.

When $X$ is $r$-regular, the orthogonal complement $N_x$ of $T_x \partial X$ is the line spanned by $n(x)$.
Thus we may define $Q$ to be the quadratic form  given on $(\alpha,tn(x))\in T_x\partial X\oplus N_x = \R^d$ by
\begin{equation*}
Q_x(\alpha,tn(x))= -\II_x(\alpha) + \tr (\II_x)t^2,
\end{equation*}
whenever $\II_x$ is defined.

For a compact set $S\subseteq \R^d$, let
\begin{align*}
{S}_+(n)&=\{s\in S \mid h(S,n)=\langle s,n\rangle\},\\
{S}_-(n)&=\{s\in S \mid -h(\check{S},n)=\langle s,n\rangle\}={S}_+(-n)
\end{align*}
denote the support sets.
Define 
\begin{align*}
\II_x^+(S)&= \max\{\II_x(s)\mid s \in {S}_+(n(x))\},\\
\II_x^-(S)&= \min\{\II_x(s)\mid s \in {S}_-(n(x))\}.
\end{align*}
Here $\II_x(s)$ means $\II_x(\pi_x(s))$ where $\pi_x:\R^d \to T_x\partial X$ is the projection. 
Since $S_+(n)$ may contain more than one point, $\II_x^+(S)$ may not attain its value at a unique $s\in S$. Thus we need the following:


%

\begin{lemma} \label{pm}
For a finite set $S\subseteq \R^d$, there exist two measurable functions $s^+,s^- : \partial X \to S$ such that $ s^{\pm}(x)\in {S}_{\pm}(n(x))$ and $\II_x^\pm(S)=\II_x(s^\pm(x))$ for all $x\in \partial X$ where $\II_x$ is defined. In particular, $\II^{\pm}(S) $ are measurable functions. 
\end{lemma}

\begin{proof}
The finitely many sets 
\begin{equation*}
\{x\in \partial X \mid s\in {S}_+(n(x)) \}\cap \{x\in \partial X \mid \II^+_x(S)=\II^+_x(s)\}
\end{equation*}
 for $s\in S$ are measurable since $\II$ is measurable. They divide $\partial X$ into finitely many measurable sets of the form
\begin{equation*}
\{ x\in \partial X \mid \{s\in {S}_+(n(x)) \mid \II^+_x(S)=\II^+_x(s)\}=S_1 \}
\end{equation*}
for $S_1\subseteq S$ and we just make a constant choice of $s^+\in S_1$ on each of them. \qed
\end{proof}

Now define 
\begin{equation*}
Q^\pm_x(S)=Q_x(s^\pm(x))
\end{equation*}
and note that this is independent of the actual choice of $s^\pm$.

%

We are now ready to state the main result of this section:
\begin{theorem}\label{nonstat}
Let $X\subseteq \R^d$ be an $r$-regular set, $A\subseteq \R^d$ a bounded Borel set, and $B,W\subseteq \R^d$ two non-empty finite sets. 
Then
\begin{align}\nonumber
\lim_{a\to 0 } \bigg( {}&a^{-2}\Ha^d(\xi_{\partial X}^{-1}(A)\cap (X\ominus a\check{B})\backslash (X\oplus a\check{W}))\\ \nonumber
&-a^{-1}\int_{\partial X\cap A}(-h(B\oplus \check{W},n))^+d\Ha^{d-1}\bigg)\\\label{basisterm}
={}&\frac{1}{2}\int_{\partial X\cap A}(Q^+(B)-Q^-(W))\delta_{(B,W)}(n)d\Ha^{d-1}\\ \label{extra}
&+\frac{1}{2}\int_{\partial X\cap A} (\II^-(W)- \II^+(B))^+\mathds{1}_{\{h(B\oplus \check{W},n)=0\}} d\Ha^{d-1}.
\end{align}
\end{theorem}

This formula is a second order version of Theorem \ref{thmV1}. Note in particular how \eqref{basisterm} resembles \eqref{HaV1}. This will be even more clear later in the isotropic setting. 

The term \eqref{extra} vanishes if the surface area measure $S_{d-1}(X,\cdot)$ on $S^{d-1}$, see~\cite{schneider}, vanishes on each of the great circles $\{n\in S^{d-1}\mid \langle b-w,n \rangle=0\}$ for $b\in B, w\in W$. In particular, it vanishes for almost all rotations of $X$.

As in \cite{rataj}, the idea of the proof of Theorem \ref{nonstat} is to apply  \cite[Theorem 2.1]{last}. Define
\begin{equation*}
f_{(B,W)}(z,a)=\mathds{1}_{\{z+aB\subseteq X,z+aW\subseteq \R^2\backslash X\}}\mathds{1}_{\xi_{\partial X}^{-1}(A)}.
\end{equation*}
For a compact set $S$ we shall write $\rho(S)=\inf\{\rho >0\mid S\subseteq B(\rho)\}$. 
Then $f_{(B,W)}(a,z)$ has support in $\partial X \oplus B(r)$ whenever $a\rho(B\cup W)\leq r$.
In this case, \cite[Theorem 2.1 and Corollary 2.5]{last} yields
\begin{equation*}
\int_{\R^d} f_{(B,W)}(z,a) dz = \sum_{m=0}^{d-1} \int_{\partial X}\int_{-r}^{r}t^m f_{(B,W)}(x+tn(x),a)s_m(k(x))dt\Ha^1(dx)
\end{equation*}
where $s_m(k)$ is the $m$th symmetric polynomial in the principal curvatures  $k=(k_1,\dots,k_{d-1})$. In particular, note that $s_1(k)=\tr(II)$.

Before proving Theorem \ref{nonstat}, we state and prove a few technical lemmas for later reference. The first one is concerned with the boundary behavior of $X$ and is an easy consequence of the definition of $r$-regular sets. 

Let 
\begin{equation*}
T^r \partial X=\{(x,\alpha)\in T \partial X \mid \alpha \in T_x \partial X ,|\alpha|< r \}
\end{equation*}
 be the open $r$-disk bundle in the tangent bundle $T\partial X$.  

\begin{lemma}\label{diff}
There is a function $q: T^r \partial X \to \R$ taking $\alpha \in T_x \partial X$ to the signed distance from $x+\alpha$ to $\partial X$ along the line parallel to $n(x)$ with the sign chosen such that $x+\alpha+q(x,\alpha)n(x)\in \partial X$.  The function
\begin{equation*}
\frac{q(x,a\alpha)}{a^2}
\end{equation*}
is uniformly bounded for $x\in \partial X$, $\alpha\in T^\rho_x\partial X$, and $a\in [-\frac{r}{\rho},\frac{r}{\rho}]\backslash \{0\}$. Moreover,
\begin{equation*}
\lim_{a\to 0}\frac{q(x,a\alpha)}{a^2}=-\frac{1}{2}\II_x(\alpha)
\end{equation*}
whenever the right hand side is defined.
\end{lemma}

\begin{proof}
Let $x\in \partial X $ and let $B_i=x-rn(x)+B(r)$ and $B_o=x+rn(x)+B(r)$ denote the inner and outer ball, respectively, as in the definition of $r$-regular sets. Then  for  $\alpha \in T^r_x\partial X$, the line segment $L_\alpha=[x+\alpha-rn,x+\alpha+rn]$ contains a boundary point $y_\alpha=x+\alpha +q(x,\alpha)n$,  as it hits both $B_i $ and $\indre(B_o)$. This point must be unique, otherwise choose $\alpha_0$ with $|\alpha_0|$ minimal such that $L_{\alpha_0}$ contains two different points $p_1$ and $p_2$. One of them, say $p_1$, must have a small neighborhood not containing any $y_\alpha$ with $|\alpha|<|\alpha_0|$ and thus the normal vector $n(p_1)$ must be exactly $-\frac{\alpha_0}{|\alpha_0|}$. But then the outer ball at $p_1$ must contain $x$, which is a contradiction. Thus $q$ is well-defined.

Moreover, $a^{-2}|q(x,a\alpha)|$ is bounded by $a^{-2}(r-\sqrt{r^2-|a\alpha|^2})$ and this is bounded for $|\alpha|\leq \rho$ and $0\neq |a|\leq \frac{r}{\rho}$. 
 
It remains to determine the limit $\lim_{a\to 0}{a^{-2}}{q(x,a\alpha)}$. 
Let $x$ be a point where $n$ is differentiable. Then $\gamma(a)= x+ a\alpha + q(x,a\alpha)n(x)$ is a $C^1$ curve in $\partial X$ with $\gamma(0)=x$ and $\gamma'(a)=\alpha$. Moreover $q(x,a\alpha)=\langle n(x),\gamma(a)-x \rangle $. By l'H\^{o}pital's rule, it is enough to show that
\begin{equation*}
\lim_{a\to 0}\frac{\langle n(x),\gamma'(a) \rangle}{2a} = -\frac{1}{2}II_x(\alpha).
\end{equation*}
But this follows because
\begin{equation*}
\lim_{a\to 0}\frac{\langle n(x),\gamma'(a) \rangle}{2a}=
\lim_{a\to 0}\frac{\langle n(\gamma(0))-n(\gamma(a)),\gamma'(a) \rangle}{2a}=-\frac{1}{2}dn_x(\alpha) = -\frac{1}{2}II_x(\alpha).
\end{equation*}\qed
\end{proof}

For $x\in \partial X$ and $s\in \R^d$ with $a|s|\leq r$, observe that for $t\in [-r,r]$,
\begin{equation*}
x+tn(x)+as \in X \textrm{ if and only if } t \leq -a\langle s,n(x)\rangle + q(x, as-\langle as,n(x)\rangle n(x)).
\end{equation*}
Thus we write
\begin{equation*}
t(as)= -a\langle s,n(x)\rangle + q(x, as-a\langle s,n(x)\rangle n(x)).
\end{equation*}
For a finite set $S$, let 
\begin{align*}
t_{-}(aS)&=\max\{t(as)\mid s\in S\}\\
t_{+}(aS)&=\min\{t(as)\mid s\in S\}. 
\end{align*}
With this notation, we obtain for $a\rho(B\cup W)<r$:
\begin{align}\label{intr}
a^{-2}&\sum_{m=0}^{d-1} \int_{\partial X}\int_{-r}^{r}t^m f_{(B,W)}(x+tn,a)s_m(k(x))dt\Ha^{d-1}(dx)\\
={}&a^{-2}\sum_{m=0}^{d-1} \int_{\partial X\cap A}\frac{1}{m+1}({t_{+}(aB)}^{m+1}-{t_{-}(aW)}^{m+1})\tau_{(B,W)}  s_m(k)dtd\Ha^{d-1}\nonumber
\end{align}
where
\begin{equation*}
\tau_{(B,W)}(x,a)=\mathds{1}_{\{ {t_{+}(aB)}>{t_{-}(aW)}\}}.
\end{equation*}

The indicator function $\tau_{(B,W)}(x,a)$ may not equal $\delta_{(B,W)}(n(x))$ everywhere, but the following lemma ensures that they do not differ too much.

\begin{lemma}\label{hbegr}
Let $B$ and $W$ be two finite non-empty sets. There are constants $C$ and $\eps$ depending only on $\rho:=\rho(B\cup W)$, such that 
\begin{equation*}
| h(B\oplus \check{W},n(x))||\tau_{(B,W)}(x,a)-\delta_{(B,W)}(n(x))|\leq Ca
\end{equation*}
whenever $a<\eps$.
\end{lemma}

\begin{proof}
On the set $\{\tau_{(B,W)}(x,a)-\delta_{(B,W)}(x)\neq 0\}$,  either  $t_{-}({aW})\geq t_+(aB)$ and $h(B\oplus \check{W},n(x))<0$ or $ t_{-}({aW})<t_+(aB)$ and $h(B \oplus \check{W},n(x))\geq 0$. 

In the first case, $t_{-}({aW})\geq t_+(aB)$ and $h(B\oplus \check{W},n)<0$ implies that
\begin{equation*}
0\leq t_{-}(a{W})- t_+(aB)= -a\langle w,n\rangle +a\langle b,n\rangle + q(x,a\alpha_1)-q(x,a\alpha_2)
\end{equation*}
for some choice of $w\in W$ and $b\in B$ and $\alpha_1,\alpha_2\in T^{\rho}_x\partial X$. Thus
\begin{align*}
0&\leq -ah(\check{W},n)-ah(B,n) \leq a\langle w,n\rangle -a\langle b,n\rangle\\ 
&\leq q(x,a\alpha_1)-q(x,a\alpha_2)\leq 2\su\{|q(x,a\alpha)| ,\, |\alpha|\leq \rho \}. 
\end{align*}
By Lemma \ref{diff}, the latter is bounded by $Ca^2$ for some constant $C$ and $a$ sufficiently small.

In the second case, let  $b\in B_+(n)$ and $w\in W_-(n)$. The claim then follows from the inequality
\begin{equation*}
 0\geq t_{-}(a{W})- t_+(aB)\geq t(aw)-t(ab)= h(B \oplus \check{W},n) + q(x,a\alpha_1)-q(x,a\alpha_2).
\end{equation*}\qed
\end{proof}

It may be that $t_{\pm}(S)\neq t(s^{\pm})$, where $s^{\pm}$ are the functions from Lemma~\ref{pm}. Thus we need the following:

\begin{lemma}\label{t'}
Let $S$ be a finite set. For each $x$, there is an $\eps>0$ such that for all $a\leq \eps$, there are  $s_\pm\in {S}_\pm(n(x)) $ with
\begin{align}
\begin{split}\label{ter}
t_{+}(aS)&=t(as_+)=-ah(S,n)+q(x,a\alpha_1)\\
t_{-}(aS)&=t(as_-)=ah(\check{S},n)+q(x,a\alpha_2)
\end{split}
\end{align}
for some $|\alpha_1| ,|\alpha_2|\leq \rho(S)$. Moreover, there is a constant $M$ depending only on $\rho(S)$ such that
\begin{equation*}
|t_{+}(aS)+ah(S,n)|, |t_{-}(aS)-ah(\check{S},n)|\leq a^2M.
\end{equation*} 
There is also a constant $M'$ not depending on $x$ such that 
\begin{equation*}
\nu(\{R \in SO(d)\mid  \exists s_{\pm}\in ({RS})_{\pm}(n(x))\text{ such that } t_{\pm}(aRS)\neq t(as_{\pm})\})\leq M'a
\end{equation*}
where $\nu$ denotes the Haar measure on $SO(d)$. 

If $B,W\subseteq \R^d$ are two finite non-empty sets, there are constants $M''$ and $\eps'>0$ depending only on $\rho(B\cup W)$, such that 
\begin{equation*}
\nu(R\in SO(d)\mid \tau_{(RB,RW)}(x,a)\neq \delta_{(RB,RW)}(n),\, h(RB\oplus R\check{W},n)\neq 0) \leq M''a
\end{equation*}
whenever $a<\eps'$.
\end{lemma}

\begin{proof}
Suppose there is an $s\in S$ with $t_{-}(aS)=t(as)\geq t(as^{-})$. This implies that $\langle s^{-},n\rangle \leq \langle s,n\rangle$ and thus
\begin{align*}
0&\leq t(as)-t(as^{-}) \\
&= -a\langle s ,n\rangle +a\langle   s^{-},n\rangle +q(x,a\alpha_1)-q(x,a\alpha_2) \\ 
&\leq q(x,a\alpha_1)-q(x,a\alpha_2)
\end{align*}
with $|\alpha_1|,|\alpha_2|\leq \rho(S)$. 
It follows that 
\begin{equation}\label{seq}
0\leq  a\langle (s - s^{-}),n\rangle \leq q(x,a\alpha_1)-q(x,a\alpha_2) \leq M_1a^2.  
\end{equation}
If this holds for arbitrarily small $a$, $\langle (s - s^{-}),n\rangle$ must equal 0 and hence $-h(\check{S},n) = \langle s ,n\rangle$. The first claim now follows by the finiteness of $S$. 

The second claim follows from \eqref{seq} because
\begin{equation*}
| t(as)+\langle as^-,n \rangle|\leq | t(as)+\langle as,n \rangle| +|\langle a(s - s^{-}),n\rangle| \leq Ma^2
\end{equation*}
for some $M$.

Furthermore,  by \eqref{seq}
\begin{align*}
\{ R \in SO(d)&\mid \exists s\in ({RS})_{-}(n) : t_{-}(aRS)\neq t(as) \}
\\ &\subseteq 
\{ R \in SO(d)\mid \exists s_1\neq s_2\in S : \langle (Rs_1-Rs_2), n \rangle  \leq M_1 a\} 
\end{align*}
and hence
\begin{align}\nonumber
\nu(R \in SO(d)&\mid \exists s\in ({RS})_{-}(n) : t_{-}(aRS)\neq t(as) )\\\nonumber
&\leq 
\nu( R \in SO(d)\mid \exists s_1\neq s_2\in S: \langle R(s_1-s_2), n \rangle  \leq M_1 a)\\ \label{numaal}
&\leq |S|^2\Ha^{d-1}(u\in S^{d-1}\mid \langle u,n \rangle \leq M_2a)\\& \leq M'a\nonumber
\end{align}
where $|S|$ is the cardinality of $S$ and $M_1$ and $M_2$ are some constants.

The case of $S_+$ is similar.

For the last claim, Lemma \ref{hbegr} shows that 
\begin{align*}
\{R\in {}& SO(d)\mid \tau_{(RB,RW)}\neq \delta_{(RB,RW)},\, h(RB\oplus R\check{W},n)\neq 0\}\\
 &\subseteq \{ R\in  SO(d)\mid | h(B\oplus\check{W}, R^{-1} n)|  \in (0,Ca] \} \\
 &\subseteq \{R\in  SO(d)\mid  \exists b\in B,w\in W, b \neq w:|\langle b-w, R^{-1} n \rangle|\leq Ca\}. 
\end{align*}
The claim follows as in \eqref{numaal}. \qed
\end{proof}

We are finally ready to prove the main theorem of this section:
\begin{proof}[Theorem \ref{nonstat}]
We must compute the limit of \eqref{intr} when $a$ tends to zero. 

First consider the terms with $m\geq 1$. 
By Lemma \ref{diff}, the terms  
\begin{equation*}
a^{-2}t(as)^{m+1}=a^{-2}(-a\langle s,n\rangle + q(x,a\alpha))^{m+1}
\end{equation*}
are bounded by some uniform constant for all $s\in B\cup W$. When $m+1>2$ they all converge to zero pointwise. Hence by  Lebesgue's theorem of dominated convergence,
\begin{equation*}
\lim_{a\to 0}a^{-2} \int_{\partial X\cap A}\frac{1}{m+1} (t_+(aB)^{m+1}-t_-(aW)^{m+1})\tau_{(B,W)} s_m(k)dt d\Ha^{d-1} =0.
\end{equation*}

For $m=1$, Lebesgue's theorem yields
\begin{align}\nonumber 
\lim_{a\to 0} {}&a^{-2}\int_{\partial X\cap A}\frac{1}{2}(t_+(aB)^{2}-t_-(aW)^{ 2})\tau_{(B,W)}s_1(k)d\Ha^1\\ 
\label{l=1} &= \int_{\partial X\cap A} \lim_{a\to 0} a^{-2}\frac{1}{2}(t_+(aB)^{2}-t_-(aW)^{ 2})\tau_{(B,W)}s_1(k)d\Ha^1\\
&=\int_{\partial X\cap A} \lim_{a\to 0}\frac{1}{2}(h(B,n)^{2}-h(\check{W},n)^{ 2})\tau_{(B,W)}s_1(k)d\Ha^1\nonumber\\
&=\int_{\partial X\cap A} \frac{1}{2}(h(B,n)^{2}-h(\check{W},n)^{ 2})\delta_{(B,W)}(n)s_1(k)d\Ha^1\nonumber
\end{align}
where the second equality uses the first part of Lemma \ref{t'} and the last equality follows since
\begin{equation*}
|h(B,n)^{2}-h(\check{W},n)^{ 2})(\tau_{(B,W)}(x,a)-\delta_{(B,W)}(n))|\leq \rho(B\oplus \check{W}) C a 
\end{equation*}
by Lemma \ref{hbegr}.

It remains to handle the $m=0$ term. Consider 
\begin{align*}
\lim_{a \to 0}  \bigg({}&\int_{\partial X\cap A} a^{-2}(t_+(aB)-t_-(aW))\tau_{(B,W)} d\Ha^{d-1}\\
&+a^{-1}\int_{\partial X\cap A}h(B\oplus\check{W},n)\delta_{(B,W)}(n) d\Ha^{d-1}\bigg)\\
={}&\lim_{a \to 0}\bigg(  \int_{\partial X\cap A}a^{-2} (t_+(aB)-t_-(aW)+ah(B\oplus\check{W},n))\delta_{(B,W)}(n) d\Ha^{d-1} \\&+ \int_{\partial X\cap A} a^{-2} (t_+(aB)-t_-(aW))(\delta_{(B,W)}(n)-\tau_{(B,W)}(x,a)) d\Ha^{d-1}\bigg).
\end{align*}
The integrand in the last line is bounded by \eqref{ter} in Lemma \ref{hbegr}, so we may apply Lebesgue's theorem. Write
\begin{equation*}
\tau_{(B,W)}(x,a)=\tau_{(B,W)}(x,a)(\mathds{1}_{\{h(B\oplus \check{W},n)>0\}} + \mathds{1}_{\{h(B\oplus \check{W},n)=0\}} +\delta_{(B,W)}(x)).
\end{equation*} 
The first term converges to zero and the last term converges to $\delta_{(B,W)}(x)$. On the set $\{h(B\oplus \check{W},n)=0\}$, 
\begin{align*}
a^{-2}{}&(t_+(aB)-t_-(aW)) \tau_{(B,W)}(x,a)\\& = a^{-2}((t_+(aB)+ah(B,n))-(t_-(aW)-ah(\check{W},n)))^+ 
\end{align*}
so the second integral converges to 
\begin{equation}\label{limtec}
-\frac{1}{2}\int_{\partial X\cap A} (\II^+(B)-\II^-(W))^+\mathds{1}_{\{h(B\oplus \check{W},n)=0\}} d\Ha^{d-1}.
\end{equation}
This follows from the first part of Lemma \ref{t'} and Lemma \ref{diff} because 
\begin{align*}
\lim_{a \to 0} a^{-2}(t_+(aB)+ah(B,n)){}&=\lim_{a \to 0} a^{-2}\min\{t(ab)+a\langle b,n\rangle \mid b\in {B}_+(n(x))\}\\
&=\min\{\lim_{a \to 0} a^{-2}(t(ab)+a\langle b,n\rangle) \mid b\in {B}_+(n(x))\}\\
&=\min\big\{-\tfrac{1}{2}\II_x(b)\mid  b\in {B}_+(n(x))\big\} \\
&=-\tfrac{1}{2}\II_x^+(B)
\end{align*}
whenever $\II_x$ is defined, and the $W$ terms are similar.

Finally,  
\begin{align*}
 a^{-2}|t_+(aB)+ah(B,n)|,\, a^{-2}|t_-(aW)-ah(\check{W},n)|
\end{align*}
are uniformly bounded by Lemma \ref{t'}, so by Lebesgue's theorem
\begin{align}\nonumber
\lim_{a \to 0} {}&a^{-2} \int_{\partial X \cap A} (t_+(aB)-t_-(aW)+a(h(B,n)+h(\check{W},n))) \delta_{(B,W)}(n)d\Ha^{d-1}\\ 
={}& \int_{\partial X\cap A}  \frac{1}{2}(\II^+(W)-\II^-({B})) \delta_{(B,W)}(n) d\Ha^{d-1}. \label{l=0}
\end{align}
The claim now follows by combining \eqref{l=1}, \eqref{limtec}, and \eqref{l=0}.
\qed
\end{proof}

\section{Hit-or-miss transforms in a rotation invariant setting}\label{rotinv}
In this section we prove a version of Theorem \ref{nonstat} where a uniform random rotation $R\in SO(d)$ is applied to the sets $B,W$. For this we let $SO(d)$ be the group of rotations of $\R^d$ and $\nu_d$ the Haar measure on $SO(d)$.

\begin{theorem}\label{isotropy}
Let $X\subseteq \R^d$ be an $r$-regular set, $A\subseteq \R^d$ a bounded Borel set, and $B,W\subseteq \R^d$ two non-empty finite sets. 
Then
\begin{align*}
\lim_{a\to 0 } \bigg( {}& a^{-2}\int_{SO(d)}\Ha^d(\xi_{\partial X}^{-1}(A)\cap (X\ominus aR\check{B})\backslash (X\oplus aR\check{W}))\nu_d(dR)
\\&-a^{-1}\Ha^{d-1}({\partial X\cap A})\int_{S^{d-1}}(-h(B\oplus \check{W},n))^+dn\bigg)
\\={}&\frac{1}{2}\int_{\partial X\cap A}\int_{SO(d)}(Q^+(RB)-Q^-(R\check{W}))\delta_{(RB,RW)}(n)\nu_d(dR)d\Ha^{d-1}.
\end{align*}
If $X$ is a smooth manifold, then the convergence is $O(a)$.
\end{theorem}

For simplicity, we write
\begin{equation*}
I= a^{-2}\int_{SO(d)}\Ha^d(\xi_{\partial X}^{-1}(A)\cap (X\ominus aR\check{B})\backslash (X\oplus aR\check{W}))\nu_d(dR)
\end{equation*}
in the following. 

For a finite set $S$, let 
\begin{equation*}
D(S)=S^{d-1}\cap \bigcup_{s_1,s_2\in S} \{n\in \R^d \mid \langle s_1,n \rangle=\langle s_2,n \rangle\}.
\end{equation*}
Then $D(S)$ has $\Ha^{d-1}$-measure zero in $S^{d-1}$. 

Whenever $n\notin D(S)$,  the two sets ${S}_\pm(n)$ contain exactly one point each.
Thus we may define $p_S^+,p_S^-:S^{d-1} \to S$ to be the unique functions such that $p_S^{\pm}(n)\in {S}_{\pm}(n)$ for $n\in S^{d-1}\backslash D(S)$ and $p_S^{\pm}(n)=0$ otherwise. These satisfy $p_S^\pm(n(x))=s^\pm(x)\mathds{1}_{\{n(x)\notin D(S)\}}$ and for $R\in SO(d)$, $p_{RS}^\pm(n) = c_Rp_S^\pm (n)$ where $c_Rp_S^\pm$ denotes the conjugation $c_Rp_S^\pm(n)=Rp_S^\pm(R^{-1}n)$.

Let
\begin{align*}
E(S){}&=\{(x,R)\in \partial X\times SO(d)\mid n(x) \in D(RS)\}\\
&=\{(x,R)\in \partial X\times SO(d)\mid R^{-1}n(x) \in D(S)\}.
\end{align*} 
Then this is also a set of measure zero. 

\begin{proof}
First note that by Tonelli's theorem
\begin{align*}
\int_{SO(d)}{}&\int_{\partial X \cap A}(-h(RB\oplus R\check{W},n))^+d\Ha^{d-1}\nu_d(dR)\\
&=\int_{\partial X \cap A}\int_{SO(d)}(-h(B\oplus \check{W},R^{-1}n))^+\nu_d(dR)d\Ha^{d-1}\\
&=\Ha^{d-1}({\partial X\cap A})\int_{S^{d-1}}(-h(B\oplus \check{W},n))^+dn .
\end{align*}
Thus, in order to prove the first statement, we must compute the limit of
\begin{align*}
I-&{}a^{-1}\lim_{a\to 0}aI \\={}&a^{-2}\sum_{m=0}^{d-1} \int_{SO(d)}\int_{\partial X\cap A}\bigg( \int_{t_-(aR W)}^{t_+(aR B)}t^m f_{(RB,RW)}(x+tn,a) s_m(k(x))dt\\&-a(-h(RB\oplus R\check{W},n(x)))^+\bigg) \Ha^{d-1}(dx)\nu_d(dR)
\end{align*}
as $a$ tends to zero. This is done exactly as in the proof of Theorem \ref{nonstat}. 
The only difference is that one has to check that the limit also commutes with the integration over $SO(d)$, but this follows because the constants bounding the integrands are also uniform with respect to the $SO(d)$-action, depending only on $\rho(B,W)=\rho(RB,RW)$. This yields the limit in Theorem \ref{isotropy} plus the term
\begin{equation}\label{vanterm}
-\frac{1}{2}\int_{\partial X\cap A} \int_{SO(d)}(\II^+(RB)-\II^-(RW))^+\mathds{1}_{\{h(RB\oplus R\check{W},n)=0\}}\nu_d(dR) d\Ha^{d-1}.
\end{equation}
But 
\begin{align*}
\{x\in {}&\partial X, \, R\in SO(d)\mid h(RB\oplus R\check{W},n)=0\}\\\subseteq {}&\{x\in \partial X, \,R\in SO(d)\mid R^{-1}n \in D(B\cup W) \} \\{}&\cup \{x\in \partial X, \, R\in SO(d)\mid p_{RB}^+(n)=p_{RW}^-(n),\, R^{-1}n \notin D(B\cup W) \}.
\end{align*}
The first set of the union has measure zero, while on the second set 
\begin{equation*}
(\II_x^+(RB)-\II_x^-(RW))^+= (\II_x(p_{RB}^+(n))-\II_x(p_{RW}^-(n)))^+=0,
\end{equation*}
hence \eqref{vanterm} vanishes.

To prove the last statement, consider
\begin{align*}
 a^{-1}&I-a^{-2}\lim_{a\to 0}aI-a^{-1}\lim_{a\to 0}(I-a^{-1}\lim_{a\to 0}aI))\\
 = &\int_{SO(d)}\int_{\partial X\cap A}\bigg(\sum_{m=0}^{d-1}\frac{a^{-3}}{m+1}({t_+(aR B)}^{m+1}-{t_-(aR W)}^{m+1})\tau_{(RB,RW)}s_m(k)\\ 
 &-  \bigg(a^{-2}\left(h(R B,n)+ h(R \check{W},n)\right)- a^{-1}\frac{1}{2}\left(\II^+(R B) - \II^-(R{W})\right)\\ 
 &+  a^{-1}\frac{1}{2} \left(h(RB,n)^2- h(R \check{W},n)^2\right)s_1(k)\bigg) \delta_{(RB,RW)}(n)\bigg)\nu_d(dR) d\Ha^{d-1}.
\end{align*}
We must see that this is bounded when $a\to 0$. 

For $m\geq 2$, $a^{-3}t(as)^{m+1}$ is uniformly bounded for all $|s|\leq \rho(B\cup W)$ by Lemma~\ref{diff}, taking care of these terms.

 For $m\leq 1$, let 
\begin{equation*} 
T=E^c\cap(\{t_+(aR B)\neq t(a p_{R B}^+(n))\}\cup \{t_-(a R W)\neq t(ap_{R W}^-(n))\}).
\end{equation*}
where $E=E(B\cup W)$. Then 
\begin{align} \nonumber
a^{-3}({}&t_+(aR B)^{m+1}-{t_-(aR W)}^{m+1})s_m(k)\\
={}& \label{split}
a^{-3}({t(a p_{RB}^+(n))}^{m+1}-{t(ap_{R W}^-(n))}^{m+1})s_m(k)\mathds{1}_{E^c\backslash T} \\
&+a^{-3}({t_+(aR B)}^{m+1}-{t_-(aR W)}^{m+1})s_m(k)\mathds{1}_T\nonumber
\end{align}
almost everywhere.

For $m=1$, note that $a^{-3}t(as)^2\leq Ka^{-1}$ for some uniform constant $K$ whenever $|s|\leq \rho(B \cup W)$. By the last part of Lemma \ref{t'}, $a^{-1}\nu_d(T)$ is bounded and hence the following integral is uniformly bounded:
\begin{align*}
\int_{SO(d)}{}&a^{-3}(({t_+(aR B)}^{2}-{t_-(aR W)}^{2})\tau_{(RB,RW)}\\&+ a^2 (h( RB,n )^2-h(R\check{W},n )^2)\delta_{(RB,RW)})s_1(k)\mathds{1}_{T}\nu_d(dR).
\end{align*}
 Moreover,
\begin{align*} 
a^{-3}{}&({t(ap_{R B}^+(n))}^2-{t(ap_{R W}^-(n))}^2 \\&+ a^2 (h( RB,n )^2-h(R\check{W},n )^2))
 s_1(k)\tau_{(RB,RW)}\mathds{1}_{E^c\backslash T}
\end{align*}
is bounded and so is
\begin{equation*} 
a^{-1}\int_{SO(d)}( h( RB,n )^2-h(R\check{W},n )^2)s_1(k)(\delta_{(RB,RW)}-\tau_{(RB,RW)})\mathds{1}_{E^c\backslash T}\nu_d(dR)
\end{equation*}
by Lemma \ref{t'}.
This takes care of the remaining term in \eqref{split}.

Finally, consider the case $m=0$. 
By Lemma \ref{t'}, 
\begin{equation*}
a^{-2}\left({t_+(aR B)}+ah(RB,n)+a^2\tfrac{1}{2}\II^+(RB)\right)
\end{equation*}
is uniformly bounded. 
Thus 
\begin{equation*}
\int_{SO(d)}a^{-3}\left({t_+(aR B)}+ah(RB,n)+a^2\tfrac{1}{2}\II^+(RB)\right)\tau_{(RB,RW)}\mathds{1}_{T}\nu_d(dR)
\end{equation*}
is bounded by the last part of Lemma \ref{t'}. A similar argument applies to the terms involving $W$ and finally
\begin{align*}
(-a^{-1}{}& h(RB,n)+\frac{1}{2}\II^+(RB)\\
&-a^{-1}h(R\check{W},n)-\frac{1}{2}\II^-(RW))(\delta_{(RB,RW)}-\tau_{(RB,RW)})\mathds{1}_{T}
\end{align*}
is bounded by Lemma \ref{hbegr} and hence the integral over $SO(d)$ belongs to $O(a)$, again by Lemma \ref{t'}.

To deal with the remaining term in \eqref{split}, we need the smoothness of $X$.
Since $X$ is smooth, $q:T^r\partial X \to \R$ is a smooth map. In local coordinates on $\partial X$,
\begin{equation*}
q(x,a\alpha)= -\tfrac{1}{2}\II_x(a\alpha) + O(|a\alpha|^3)
\end{equation*}
where the $O(|a\alpha|^3)$ term is bounded by 
\begin{equation*}
C|a\alpha|^3\su \left\{\left|\frac{\partial^3q}{d\alpha_id\alpha_jd\alpha_k }(x,a\alpha)\right|,\, i,j,k=1,\dots,d-1, |a\alpha|\leq r \right\}.
\end{equation*}
The functions $\frac{\partial^3q}{d\alpha_id\alpha_jd\alpha_k}(x,a\alpha)$ are continuous and hence bounded on compact sets. Since $\partial X \cap A$ is contained in a union of finitely many compact sets contained in coordinate neighborhoods, the whole $O(|a\alpha|^3)$ term is uniformly bounded on $T^r\partial X_{\mid A}$ by $C'a^3$ for some constant $C'$.

This shows that $a^{-3}({t(ap_{R B}^+(n))}+ah(RB,n)+ a^2\frac{1}{2}\II(p_{R B}^+(n)))$ is bounded and that the corresponding statement is true for $W$, so it remains to consider
\begin{align}\label{sidste}
(-a^{-1}{}& h(RB\oplus R\check{W},n)+\tfrac{1}{2}(\II^+(RB)-\II^-(RW)))\\&\times (\delta_{(RB,RW)}-\tau_{(RB,RW)})\mathds{1}_{E^c\backslash T}.\nonumber
\end{align}
If $h(RB\oplus R\check{W},n)=0$, then $p_{R B}^+(n)=p_{R W}^+(n)$ since $(x,R)\in E^c$ and thus \eqref{sidste} vanishes. It follows from the last part of Lemma \ref{t'} that the integral of \eqref{sidste} over all of $SO(d)$ belongs to $O(a)$.
\qed
\end{proof}

The formula of Theorem \ref{isotropy} may be simplified further:
\begin{theorem}
Let $X,A,B,W\subseteq \R^d$ be as in  Theorem \ref{isotropy}. Then
\begin{align*}
\lim_{a\to 0} (I-a^{-1}\lim_{a\to 0}aI) 
={}&\frac{1}{2} C_{d-2}(X; A) \int_{S^{d-1}}\Big( d(h(B,n)^2  - h(\check{W},n)^2)\\
& -(|p_B^+(n)|^2-|p_W^-(n)|^2 ) \Big) \delta_{(B,W)}(n)\Ha^{d-1}(dn).
\end{align*}
where $C_{d-2}(X;\cdot)$ is the $(d-2)$th curvature measure on $X$ normalized as in~\cite{schneider}.
\end{theorem}

In particular, we recover $C_{d-2}(X; A)$ up to a constant depending only on the sets $B$ and $W$.

\begin{proof}
For a finite set $S$ and $x\in \partial X$ fixed, we compute
\begin{align*}
\int_{SO(d)}{}&Q^+_x(R S)\delta_{(RB,RW)}(n)\nu_d (dR)\\ 
={}& \int_{SO(d)} \int_{SO(d-1)}Q^+_x(PR S) \delta_{(B,W)}((PR)^{-1}n)\nu_{d-1}(dP) \nu_d(dR)  \\
={}& \int_{SO(d)} \int_{SO(d-1)}Q_x(c_{PR}p_S^+(n)) \nu_{d-1}(dP) \delta_{(B,W)}(R^{-1}n)\nu_d(dR)
\end{align*}
where $SO(d-1)$ is the subgroup that keeps $n$ fixed. Note that $c_{PR}p_S^+=P c_Rp_S^+$. Hence
\begin{align*}
\int_{SO(d)}{}&Q^+_x(R S)\delta_{(B,W)}(R^{-1}n)\nu_d(dR) \\
={}& \int_{SO(d)} \int_{SO(d-1)}Q_x(P c_R p_S^+(n))\nu_{d-1}(dP) \delta_{(B,W)}(R^{-1}n)  \nu_d(dR) \\
={}& \int_{SO(d)} \bigg(\int_{SO(d-1)}(-\II_x(P c_Rp_S^+(n))) \nu_{d-1}(dP) \\ 
&\qquad+ \tr(\II_x)  \langle c_R p_S^+(n),n\rangle^2\bigg)\delta_{(B,W)}(R^{-1}n)\nu_{d}(dR) \\
={}&\int_{SO(d)} \bigg(\frac{1}{d-1}\tr(\II_x)( \langle c_Rp_S^+(n),n\rangle^2-|c_Rp_S^+(n)|^2)\\ 
&\qquad + \tr(\II_x) \langle c_Rp_S^+(n),n\rangle^2 \bigg)\delta_{(B,W)}(R^{-1}n)\nu_d( dR) \\
={}&\int_{SO(d)} \frac{1}{d-1}\tr(\II_x)(  d\langle p_S^+(R^{-1} n),R^{-1} n\rangle^2\\
&\qquad-|p_S^+(R^{-1}n)|^2) \delta_{(B,W)}(R^{-1}n)\nu_d(dR)\\
={}&\int_{S^{d}} \frac{1}{d-1}\tr(\II_x)(  d h(S,u)^2-|p_S^+(u)|^2)\delta_{(B,W)}(u) \Ha^{d-1}(du).
\end{align*}
The third equality here may be proved using the characterization of the trace as the unique basis invariant linear map on the space of linear maps on $\R^{d-1}$.
Inserting the above in Theorem~\ref{isotropy} yields the formula.\qed
\end{proof}

\section{Application to configurations}\label{apply}
We now return to the design based setting where we observe a compact $r$-regular set $X\subseteq \R^d$ on a random lattice $\La$.

We introduce the following notation:
\begin{align*}
\bar{\varphi}_j(X)={}&\sum_{l:\xi_l\in \eta_j^d}\int_{\partial X} (-h(B_l\oplus \check{W}_l,n(x)))^+\Ha^{d-1}(dx),\\
\bar{\psi}_j={}&2\sum_{l:\xi_l\in \eta_j^d} \int_{S^{d-1}} (-h(B_l\oplus \check{W}_l,n))^+\Ha^{d-1}(dn),\\
\lambda_l(X)={}&\frac{1}{2}\int_{\partial X}(Q^+(B_l)-Q^-({W}_l))\delta_{(B_l,W_l)}(n)d\Ha^{d-1}\\
&-\frac{1}{2}\int_{\partial X} ( \II^+(B_l)-\II^-(W_l))^+\mathds{1}_{\{h(B_l\oplus \check{W}_l,n)=0\}} d\Ha^{d-1},\\
\bar{\lambda}_j(X)={}&\sum_{l:\xi_l\in \eta_j^d}\lambda_l(X),\\
\mu_l={}&\frac{\pi}{d-1} \int_{S^{d-1}}\big(d(h(B_l,n)^2  - h(\check{W}_l,n)^2) \\
&\qquad -(|p_{B_l}^+(n)|^2-|p_{W_l}^-(n)|^2 ) \big)\delta_{(B_l,W_l)}(n)dn,\\
\bar{\mu}_j={}&\sum_{l:\xi_l\in \eta_j^d}\mu_l.
\end{align*}
Combining the observation \eqref{HEN} with Theorem \ref{nonstat} and \ref{isotropy}, we obtain:

\begin{corollary}\label{confcor}
Let $\xi_l$ be a configuration with black and white points $(B_l,W_l)$.
If $\La$ is a stationary non-isotropic lattice,
\begin{equation*}
\lim_{a\to 0} (a^{d-2}EN_l-a^{-1}\lim_{a\to 0} a^{d-1}EN_l)=\lambda_l(X).
\end{equation*}
If $\La$ is stationary isotropic,
\begin{equation*}
\lim_{a\to 0} (a^{d-2}EN_l-a^{-1}\lim_{a\to 0} a^{d-1}EN_l)= \mu_l V_{d-2}(X).
\end{equation*}
In particular, suppose $\hat{V}_{d-2}$ is a local estimator of the form \eqref{motion}.
In both cases $\lim_{a\to 0}E\hat{V}_{d-2}(X)$ exists if and only if $\lim_{a\to 0} aE\hat{V}_{d-2}(X)=0$, where
\begin{align}\nonumber
\lim_{a\to 0} aE\hat{V}_{d-2}(X)&= \sum_{j\in J}w_j^{(d-2)} \bar{\varphi}_j(X)\\
\lim_{a\to 0} aE\hat{V}_{d-2}(X)&= V_{d-1}(X) \sum_{j\in J}w_j^{(d-2)} \bar{\psi}_j \label{inth2}
\end{align}
in the non-isotropic and isotropic case, respectively.
In this case, the limit is 
\begin{equation*}
\lim_{a\to 0} E\hat{V}_{d-2}(X)=\sum_{j\in J}w_j^{(d-2)} \bar{\lambda}_j(X)
\end{equation*}
in the non-isotropic case, 
and in the isotropic case
\begin{equation}
\lim_{a\to 0} E\hat{V}_{d-2}(X)= V_{d-2}(X)\sum_{j\in J}w_j^{(d-2)} \bar{\mu}_j.\label{isobias}
\end{equation}
\end{corollary}

In the isotropic case, there are some symmetries allowing us to reduce the above formula a bit further. The following properties are obvious:
\begin{proposition}\label{muer}
\begin{equation*}
\mu_l=-\mu_{(2^{2^d}-1-l)}.
\end{equation*}
If $\xi_{l_1}$ and $\xi_{l_2}$ belong to the same configuration class,
\begin{equation*}
\mu_{l_1}=\mu_{l_2}.
\end{equation*}
\end{proposition}
Let $\xi_l\in \eta_{j_1}^d$ and let $\eta_{j_2}^d $ be the configuration class of $\xi_{(2^{2^d}-1-l)}$. Then by the corollary, we may as well choose $w_{j_1}^{(d-2)}=-w_{j_2}^{(d-2)}$. Since $\bar{\psi}_{j_1}=\bar{\psi}_{j_2}$, this also ensures that the asymptotic mean exists. Finally it ensures that interchanging foreground and background changes the sign of $\hat{V}_{d-2}$, which is desirable since $V_{d-2}$ has this property.
 
Moreover, not all $\mu_l$ are zero, e.g.\ $\mu_1>0$. If $\eta_1^d$ and $\eta_{2^d-1}^d$ denote the configuration classes of $\xi_1$ and $\xi_{2^{2^d}-2}$, respectively, this shows:
\begin{corollary}
In the isotropic case, asymptotically unbiased estimators for $V_{d-2}$ do exist. For instance, the estimator with all weights equal to zero except
\begin{equation*}
w_{1}^{(d-2)}=-w_{2^{d}-1}^{(d-2)}=\frac{1}{2\bar{\mu}_1}
\end{equation*}
is asymptotically unbiased. 
\end{corollary}

The last proposition of this section reduces the formula for $\bar{\mu}_j$ in a way that resembles \eqref{HaV1} and the formula for $\bar{\psi}_j$ even more.
\begin{proposition}
\begin{equation*}
\bar{\mu}_j=
\frac{d\pi}{d-1} \sum_{l:\xi_l\in \eta_j^d}\int_{S^{d-1}} (h(B_l,n)^2-h(\check{W_l},n)^2)\delta_{(B_l,W_l)}(n)\Ha^{d-1}( dn).
\end{equation*}
\end{proposition}

\begin{proof} 
Choose a rotation $R$ taking $C$ to $\check{C}$.
For each configuration $\xi_l$ we let $\xi_{l'}=R(\xi_l)+(1,1,1)$. 
Then
\begin{align*}
|p_{B_l}^+(n)|^2&= d- |p_{B_{l'}}^+(R n)|^2,\\
|p_{W_l}^-(n)|^2&= d- |p_{W_{l'}}^-(R n)|^2,
\end{align*}
and $\delta_{(B_l,W_l)}(n)=\delta_{(B_{l'},W_{l'})}(R n)$, so that 
\begin{align*}
\int_{S^{d-1}}{}&\left((|p_{W_l}^-|^2-|p_{B_l}^+|^2)\delta_{(B_l,W_l)}+(|p_{W_{l'}}^-|^2-|p_{B_{l'}}^+|^2)\delta_{(B_{l'},W_{l'})}\right) d\Ha^{d-1} \\ ={}&
\int_{S^{d-1}}(d-d) \delta_{(B_l,W_l)} d\Ha^{d-1}=0.
\end{align*}
Hence
\begin{align*}
\mu_l+\mu_{l'}=
\frac{\pi d}{d-1}{} &\int_{S^{d-1}} \Big((h(B_l,n)^2-h(\check{W_l},n)^2)\delta_{(B_l,W_l)}(n)\\ 
&  +(h(B_{l'},n)^2-h(\check{W_{l'}},n)^2)\delta_{(B_{l'},W_{l'})}(n)\Big)\Ha^{d-1}( dn)
\end{align*}
from which the claim follows.\qed
\end{proof}

\section{More on the isotropic setting in 3D}\label{3Diso}
We now specialize to the isotropic situation. That is, we assume throughout this section that $X\subseteq \R^3$ is an $r$-regular compact set observed on a stationary isotropic lattice $a\La$. Theorem  \ref{confcor} determines the set of all asymptotically unbiased estimators for $V_{d-2}$ as follows: an estimator is asymptotically unbiased if and only if the weights satisfy two linear equations 
\begin{gather*}\nonumber
\sum_{j\in J}w_j^{(d-2)}\bar{\psi}_j=0,\\
\sum_{j\in J}w_j^{(d-2)}\bar{\mu}_j=1.\label{unbi cond}
\end{gather*}
The first one ensures that the asymptotic mean exists and the second one makes the estimator asymptotically unbiased.

The coefficients $\bar{\psi}_j$ and $\bar{\mu}_j$ can in principle be computed directly for each configuration. However, the actual computations are tedious. The computations in dimension $d=2$ were done in \cite{am}. Below we consider the case $d=3$.

First note that $\delta_{(B_l,W_l)}$ vanishes if $W_l$ and $B_l $ cannot be strongly separated by a hyperplane, so we may ignore such configurations. Recall that we also ignore the configurations $\xi_0$ and $\xi_{255}$. The remaining configurations fall into one of the eight equivalence classes pictured below:
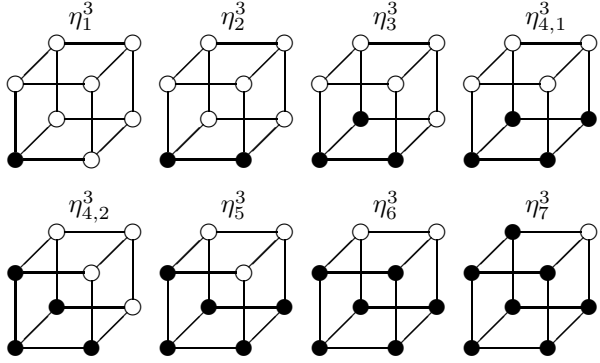
\begin{equation*}
\setlength{\unitlength}{0.1cm}
\begin{picture}(80,45)
\put(7,43){$\eta_1^3$}
\put(27,43){$\eta_2^3$}
\put(47,43){$\eta_3^3$}
\put(67,43){$\eta_{4,1}^3$}

\put(7,18){$\eta_{4,2}^3$}
\put(27,18){$\eta_5^3$}
\put(47,18){$\eta_6^3$}
\put(67,18){$\eta_7^3$}

\put(0,25){\circle*{2}}
\put(10,25){\circle{2}}
\put(0,35){\circle{2}}
\put(10,35){\circle{2}}
\put(5.4,30.4){\circle{2}}
\put(15.4,30.4){\circle{2}}
\put(5.4,40.4){\circle{2}}
\put(15.4,40.4){\circle{2}}
\put(0,26){\line(0,1){8}}
\put(1,25){\line(1,0){8}}
\put(10,26){\line(0,1){8}}
\put(1,35){\line(1,0){8}}

\put(5.4,31.4){\line(0,1){8}}
\put(6.4,30.4){\line(1,0){8}}
\put(15.4,31.4){\line(0,1){8}}
\put(6.4,40.4){\line(1,0){8}}

\put(0.7,25.7){\line(1,1){4}}
\put(10.7,25.7){\line(1,1){4}}
\put(0.7,35.7){\line(1,1){4}}
\put(10.7,35.7){\line(1,1){4}}
\put(20,25){\circle*{2}}
\put(30,25){\circle*{2}}
\put(20,35){\circle{2}}
\put(30,35){\circle{2}}
\put(25.4,30.4){\circle{2}}
\put(35.4,30.4){\circle{2}}
\put(25.4,40.4){\circle{2}}
\put(35.4,40.4){\circle{2}}
\put(20,26){\line(0,1){8}}
\put(21,25){\line(1,0){8}}
\put(30,26){\line(0,1){8}}
\put(21,35){\line(1,0){8}}

\put(25.4,31.4){\line(0,1){8}}
\put(26.4,30.4){\line(1,0){8}}
\put(35.4,31.4){\line(0,1){8}}
\put(26.4,40.4){\line(1,0){8}}

\put(20.7,25.7){\line(1,1){4}}
\put(30.7,25.7){\line(1,1){4}}
\put(20.7,35.7){\line(1,1){4}}
\put(30.7,35.7){\line(1,1){4}}

\put(40,25){\circle*{2}}
\put(50,25){\circle*{2}}
\put(40,35){\circle{2}}
\put(50,35){\circle{2}}
\put(45.4,30.4){\circle*{2}}
\put(55.4,30.4){\circle{2}}
\put(45.4,40.4){\circle{2}}
\put(55.4,40.4){\circle{2}}
\put(40,26){\line(0,1){8}}
\put(41,25){\line(1,0){8}}
\put(50,26){\line(0,1){8}}
\put(41,35){\line(1,0){8}}

\put(45.4,31.4){\line(0,1){8}}
\put(46.4,30.4){\line(1,0){8}}
\put(55.4,31.4){\line(0,1){8}}
\put(46.4,40.4){\line(1,0){8}}

\put(40.7,25.7){\line(1,1){4}}
\put(50.7,25.7){\line(1,1){4}}
\put(40.7,35.7){\line(1,1){4}}
\put(50.7,35.7){\line(1,1){4}}

\put(60,25){\circle*{2}}
\put(70,25){\circle*{2}}
\put(60,35){\circle{2}}
\put(70,35){\circle{2}}
\put(65.4,30.4){\circle*{2}}
\put(75.4,30.4){\circle*{2}}
\put(65.4,40.4){\circle{2}}
\put(75.4,40.4){\circle{2}}
\put(60,26){\line(0,1){8}}
\put(61,25){\line(1,0){8}}
\put(70,26){\line(0,1){8}}
\put(61,35){\line(1,0){8}}

\put(65.4,31.4){\line(0,1){8}}
\put(66.4,30.4){\line(1,0){8}}
\put(75.4,31.4){\line(0,1){8}}
\put(66.4,40.4){\line(1,0){8}}

\put(60.7,25.7){\line(1,1){4}}
\put(70.7,25.7){\line(1,1){4}}
\put(60.7,35.7){\line(1,1){4}}
\put(70.7,35.7){\line(1,1){4}}

\put(0,0){\circle*{2}}
\put(10,0){\circle*{2}}
\put(0,10){\circle*{2}}
\put(10,10){\circle{2}}
\put(5.4,5.4){\circle*{2}}
\put(15.4,5.4){\circle{2}}
\put(5.4,15.4){\circle{2}}
\put(15.4,15.4){\circle{2}}
\put(0,1){\line(0,1){8}}
\put(1,0){\line(1,0){8}}
\put(10,1){\line(0,1){8}}
\put(1,10){\line(1,0){8}}

\put(5.4,6.4){\line(0,1){8}}
\put(6.4,5.4){\line(1,0){8}}
\put(15.4,6.4){\line(0,1){8}}
\put(6.4,15.4){\line(1,0){8}}

\put(0.7,0.7){\line(1,1){4}}
\put(10.7,0.7){\line(1,1){4}}
\put(0.7,10.7){\line(1,1){4}}
\put(10.7,10.7){\line(1,1){4}}

\put(20,0){\circle*{2}}
\put(30,0){\circle*{2}}
\put(20,10){\circle*{2}}
\put(30,10){\circle{2}}
\put(25.4,5.4){\circle*{2}}
\put(35.4,5.4){\circle*{2}}
\put(25.4,15.4){\circle{2}}
\put(35.4,15.4){\circle{2}}
\put(20,1){\line(0,1){8}}
\put(21,0){\line(1,0){8}}
\put(30,1){\line(0,1){8}}
\put(21,10){\line(1,0){8}}

\put(25.4,6.4){\line(0,1){8}}
\put(26.4,5.4){\line(1,0){8}}
\put(35.4,6.4){\line(0,1){8}}
\put(26.4,15.4){\line(1,0){8}}

\put(20.7,0.7){\line(1,1){4}}
\put(30.7,0.7){\line(1,1){4}}
\put(20.7,10.7){\line(1,1){4}}
\put(30.7,10.7){\line(1,1){4}}

\put(40,0){\circle*{2}}
\put(50,0){\circle*{2}}
\put(40,10){\circle*{2}}
\put(50,10){\circle*{2}}
\put(45.4,5.4){\circle*{2}}
\put(55.4,5.4){\circle*{2}}
\put(45.4,15.4){\circle{2}}
\put(55.4,15.4){\circle{2}}
\put(40,1){\line(0,1){8}}
\put(41,0){\line(1,0){8}}
\put(50,1){\line(0,1){8}}
\put(41,10){\line(1,0){8}}

\put(45.4,6.4){\line(0,1){8}}
\put(46.4,5.4){\line(1,0){8}}
\put(55.4,6.4){\line(0,1){8}}
\put(46.4,15.4){\line(1,0){8}}

\put(40.7,0.7){\line(1,1){4}}
\put(50.7,0.7){\line(1,1){4}}
\put(40.7,10.7){\line(1,1){4}}
\put(50.7,10.7){\line(1,1){4}}

\put(60,0){\circle*{2}}
\put(70,0){\circle*{2}}
\put(60,10){\circle*{2}}
\put(70,10){\circle*{2}}
\put(65.4,5.4){\circle*{2}}
\put(75.4,5.4){\circle*{2}}
\put(65.4,15.4){\circle*{2}}
\put(75.4,15.4){\circle{2}}
\put(60,1){\line(0,1){8}}
\put(61,0){\line(1,0){8}}
\put(70,1){\line(0,1){8}}
\put(61,10){\line(1,0){8}}

\put(65.4,6.4){\line(0,1){8}}
\put(66.4,5.4){\line(1,0){8}}
\put(75.4,6.4){\line(0,1){8}}
\put(66.4,15.4){\line(1,0){8}}

\put(60.7,0.7){\line(1,1){4}}
\put(70.7,0.7){\line(1,1){4}}
\put(60.7,10.7){\line(1,1){4}}
\put(70.7,10.7){\line(1,1){4}}
\end{picture}
\end{equation*}
%
\begin{proposition} 
$\lim_{a\to 0 } aE\hat{V}_1(X)$ equals
\begin{align*} V_2(X)\Big({}&(3-4\zeta)(w_1^{(1)}+w_7^{(1)})+(-3+12\zeta-3\sqrt{2})(w_2^{(1)}+w_6^{(1)})\\
&+(3-12\zeta+6\sqrt{2}-2\sqrt{3})(w_3^{(1)}+w_5^{(1)})+(-3+2\sqrt{3})w_{4,1}^{(1)}\\&+(8\zeta-6\sqrt{2}+2\sqrt{3})w_{4,2}^{(1)}\Big)
\end{align*}
where $\zeta =3\sqrt{2}\frac{\arctan(\sqrt{2})}{2\pi}$. 
\end{proposition}

\begin{proof}
We must compute the coefficients $\bar{\psi}_j$ in \eqref{inth2}. The computations are similar to the computations of $\bar{\mu}_j$ below, so we leave them out here. \qed
\end{proof}

\begin{theorem}\label{3Dlim}
$\lim_{a\to 0} E\hat{V}_1(X)$ exists if and only if the weights satisfy
\begin{align*}\nonumber
0=\Big({}&(3-4\zeta)(w_1^{(1)}+w_7^{(1)})+(-3+12\zeta-3\sqrt{2})(w_2^{(1)}+w_6^{(1)})\\ 
&+(3-12\zeta+6\sqrt{2}-2\sqrt{3})(w_3^{(1)}+w_5^{(1)})+(-3+2\sqrt{3})w_{4,1}^{(1)}\\
&+(8\zeta-6\sqrt{2}+2\sqrt{3})w_{4,2}^{(1)}\Big)\nonumber
\end{align*}
and in this case
\begin{align*}
\lim_{a\to 0} E\hat{V}_1(X) ={}& V_1(X)\Big((3-\sqrt{3})(w_1^{(1)}-w_7^{(1)})+(3\sqrt{3}-3\sqrt{2})(w_2^{(1)}-w_6^{(1)})\\ 
& + (-3+6\sqrt{2}-3\sqrt{3})(w_3^{(1)}-w_5^{(1)})\Big).
\end{align*}
If $X$ is smooth, the convergence is $O(a)$.
\end{theorem}

\begin{proof}
By Corollary \ref{confcor} we must compute the coefficients $\bar{\mu}_j$ in \eqref{isobias}. By Proposition \ref{muer}, $\bar{\mu}_{4,1}=\bar{\mu}_{4,2}=0$ and $\bar{\mu}_j=\bar{\mu}_{8-j}$, so it is enough to compute $\bar{\mu}_j$ for $j=1,2,3$.

The hyperplanes $\langle x_{i_1},n\rangle = \langle x_{i_2},n\rangle$ with $x_{i_1},x_{i_2} \in C_0$ divide $S^{2}$ into 96 triangles of two types: 48 triangle $T^1_{\alpha \beta \gamma}$ with vertices 
\begin{equation*}
v_\alpha, \tfrac{1}{\sqrt{2}}(v_\alpha+v_\beta), \tfrac{\sqrt{2}}{\sqrt{3}}\left(v_\alpha+\tfrac{1}{2}(v_\alpha+v_\beta)\right)
\end{equation*}
 and 48 triangles $T_{\alpha \beta \gamma}^2$ with vertices 
\begin{equation*} 
\tfrac{1}{\sqrt{2}}(v_\alpha+v_\beta),\tfrac{\sqrt{2}}{\sqrt{3}}\left(v_\alpha+\tfrac{1}{2}(v_\beta+v_\gamma)\right), \tfrac{1}{\sqrt{3}}(v_\alpha+v_\beta+v_\gamma)
\end{equation*}
  where $\{|\alpha|,|\beta|,|\gamma|\}=\{1,2,3\}$ and $v_{\pm |\alpha|}=\pm e_{|\alpha|}$.

On the interior of each $T^m_{\alpha \beta \gamma}$, all indicator functions $\delta_{(B_l,W_l)}$ and functions $b_l^+$ and $w_l^-$ are constant. For each $k=1,\dots ,7$, there is exactly one configuration containing $k$ points such that $\delta_{{B}_l, W_l}$ is non-zero on $T^m_{\alpha \beta \gamma}$. For $k=4$, this configuration is of type $\eta_{4,1}^3$ on $T^1_{\alpha \beta \gamma }$ and of type $\eta_{4,2}^3$ on $T^2_{\alpha \beta \gamma}$. 

Let $R_{\alpha \beta \gamma}$ be the orthogonal map taking $(v_\alpha,v_\beta,v_\gamma)$ to $(e_\alpha,e_\beta,e_\gamma)$. This takes $T_{\alpha \beta \gamma}^m$ to $T_0^m:=T_{123}^m$
and
$h(B_l,n)=h(R_{\alpha \beta \gamma}B_l,R_{\alpha \beta \gamma }n)$.
Thus 
\begin{equation*}
\int_{T_{\alpha \beta \gamma}^m}h(B_l,n)^2\delta_{(B_l,W_l)}(n) dn=\int_{T_0^m}h(R_{\alpha \beta \gamma}B_l,n)^2\delta_{(R_{\alpha \beta \gamma}B_l,R_{\alpha \beta \gamma}W_l)}(n) dn.
\end{equation*}
There is a unique $x\in C_0$ such that $R_{\alpha \beta \gamma}C+x= C$. Each $x\in C_0$ corresponds to six different $R_{\alpha \beta \gamma}$. Since $\delta_{(R_{\alpha \beta \gamma }B_l,R_{\alpha \beta \gamma }W_l)}(n)=\delta_{(R_{\alpha \beta \gamma }B_l+x,R_{\alpha \beta \gamma}W_l+x)}(n)$,  
\begin{align*}
\bar{\mu}_j &= \pi \sum_{l: \xi_l\in \eta_j^3} 
\int_{S^{d-1}} \frac{d}{d-1}( h(B_l,n)^2- h(\check{W_l},n)^2)\delta_{(B_l,W_l)}(n)dn\\
 &= \frac{3}{2}\pi \sum_{l: \xi_l\in \eta_j^3} \sum_{\alpha \beta \gamma}
\int_{T_{\alpha \beta \gamma}^1\cup T_{\alpha \beta \gamma}^2} (h(B_l,n)^2-h(\check{W_l},n)^2)\delta_{(B_l,W_l)}(n)dn\\
 &= \frac{3}{2}\pi \sum_{l: \xi_l\in \eta_j^3}  \sum_{\alpha \beta \gamma}
\int_{T_{0}^1\cup T_{0}^2} (h(R_{\alpha \beta \gamma}B_l,n)^2 -h(R_{\alpha \beta \gamma}\check{W_l},n)^2)\\ 
&\qquad \qquad \qquad \qquad \times \delta_{(R_{\alpha \beta \gamma}B_l,R_{\alpha \beta \gamma}W_l)}(n)dn\\
 &= \frac{3}{2}\pi \sum_{x\in C_0} 
\int_{T_0^1\cup T_0^2} 6( h(B_{l_j}-x,n)^2- h(\check{W}_{l_j}+x,n)^2) \delta_{(B_{l_j},W_{l_j})}(n)dn.
\end{align*}
where $\xi_{l_j}$ is the unique configuration of type $j$ such that $\delta_{(B_{l_j},W_{l_j})}$ is not everywhere zero on $T_0^1\cup T_0^2$.

For $j=1$, $p_{B_{l_1}}^+=(0,0,0)$ and $p_{W_{l_1}}^-=(0,0,1)$ on both $T_0^1$ and $T_0^2$. From this,
\begin{align*}
\bar{\mu}_1 &= 9\pi \sum_{x\in C_0} 
\int_{T_0^1\cup T_0^2} (\langle(0,0,0)-x,n\rangle^2 -\langle(0,1,0)-x,n\rangle^2)dn\\
&=9\pi \sum_{x\in C_0} \int_{T_0^1\cup T_0^2}8(n_1+n_2)n_3dn.
\end{align*}
where $n=(n_1,n_2,n_3)$. Parametrize the sphere by $(\cos\phi,\cos\theta\sin\phi,\sin\theta\sin\phi)$ with 
$\theta \in (0,2\pi)$ and $\phi\in (0,\pi)$. Then this becomes
\begin{align*}
\bar{\mu}_1 ={}&72\pi \frac{1}{4\pi} \int_0^{\frac{\pi}{4}}\int_0^{\arccos\left(\frac{\cos\theta}{\sqrt{1+\cos^2\theta}}\right)} (\cos\theta\sin\theta\sin^3\phi+\sin\theta\sin^2\phi\cos\phi)d\phi d\theta \\={}& {3-\sqrt{3}}.
\end{align*}

For $j=2$, we get $p_{B_{l_2}}^+=(0,0,1)$ and $p_{W_{l_2}}^+=(0,1,0)$ and thus 
\begin{align*}
\bar{\mu}_2 &= 9\pi \sum_{x\in C_0} 
\int_{T_0^1\cup T_0^2} (\langle(0,0,1)-x,n\rangle^2- \langle(0,1,0)-x,n)^2\rangle dn\\
&=9\pi \sum_{x\in C_0} 
\int_{T_0^1\cup T_0^2} 8(n_2-n_3)n_1dn\\
&= 18 \int_0^{\frac{\pi}{4}}\int_0^{\arccos\left(\frac{\cos\theta}{\sqrt{1+\cos^2\theta}}\right)} (\cos\theta-\sin\theta)\cos\phi \sin^2\phi d\phi d\theta \\&= {3\sqrt{3}-3\sqrt{2}}.
\end{align*}

Finally for $j=3$, $p_{B_{l_3}}^+=(0,1,0)$ and $p_{W_{l_3}}^-=(1,0,0)$ on $T^1_0$, while on $T_0^2$, $p_{W_{l_3}}^-=(0,1,1)$. However, on both triangles
\begin{equation*}
\sum_{x\in C_0} 
(\langle p_{B_{l_3}}^+-x,n\rangle^2- \langle p_{W_{l_3}}^--x,n)^2\rangle =8(n_1-n_2)n_3.
\end{equation*}
and thus
\begin{align*}
\bar{\mu}_3 &=72\pi  \int_{T_0} (n_3-n_1)n_2 dn\\
&=18  \int_0^{\frac{\pi}{4}}\int_0^{\arccos\left(\frac{\cos\theta}{\sqrt{1+\cos^2\theta}}\right)} (\cos\phi-\cos\theta\sin\phi)\sin\theta\sin^2\phi  d\phi d\theta \\
&= {-3\sqrt{3}+6\sqrt{2}-3}.
\end{align*}
Inserting this in \eqref{isobias} proves the claim.\qed
\end{proof}

\section{Unbiased estimators for the Euler characteristic in 2D}\label{EC}
The remainder of this paper is devoted to the case where $\La$ is a stationary non-isotropic lattice. In dimension $d=2$, $V_{d-2}$ is simply the Euler characteristic. In this case, it follows from known results that there exists a unique asymptotically unbiased estimator of the form \eqref{motion}. The existence goes back to Pavlidis \cite{pavlidis} and the uniqueness follows from the results of \cite{rataj}. In this section, we show how this also follows as a consequence of Corollary \ref{confcor}. In contrast, we shall see in Section \ref{higher} that no asymptotically unbiased estimator of the form \eqref{estdef} can exist in dimensions $d\geq 3$. 

Let $X\subseteq \R^2$ be an $r$-regular set observed on a stationary lattice. Observe that the set $A=\{n\in S^1\mid h(B_l\oplus \check{W}_l,n)=0\}$ is finite. If $n(x)\in A$ and $n$ is differentiable at $x$, then either $dn=0$, in which case $\II_x=0$, or $dn\neq 0$ and thus there must be a neighborhood of $x$ where $n\notin A$. Thus \eqref{extra} vanishes in 2D.

Let $\hat{V}_{d-2}$ be a local estimator of the form \eqref{estdef}.
Again we ignore the configurations $\xi_0$ and $\xi_{15}$. Moreover, $\delta_{(B_l,W_l)}$ vanishes for $\xi_6$ and $\xi_{9}$. The remaining configurations fall into one of the following three equivalence classes:

\begin{equation*}
\setlength{\unitlength}{0.088cm}
\begin{picture}(50,16)
\put(3,14){$\eta_1^2$}
\put(23,14){$\eta_{2}^2$}
\put(43,14){$\eta_{3}^2$}

\put(0,0){\circle*{2}}
\put(10,0){\circle{2}}
\put(0,10){\circle{2}}
\put(10,10){\circle{2}}
\put(0,1){\line(0,1){8}}
\put(1,0){\line(1,0){8}}
\put(10,1){\line(0,1){8}}
\put(1,10){\line(1,0){8}}

\put(20,0){\circle*{2}}
\put(30,0){\circle*{2}}
\put(20,10){\circle{2}}
\put(30,10){\circle{2}}
\put(20,1){\line(0,1){8}}
\put(21,0){\line(1,0){8}}
\put(30,1){\line(0,1){8}}
\put(21,10){\line(1,0){8}}

\put(40,0){\circle*{2}}
\put(50,0){\circle*{2}}
\put(40,10){\circle*{2}}
\put(50,10){\circle{2}}
\put(40,1){\line(0,1){8}}
\put(41,0){\line(1,0){8}}
\put(50,1){\line(0,1){8}}
\put(41,10){\line(1,0){8}}

\end{picture}
\end{equation*}

For $d=2$, Theorem \ref{nonstat} reduces to:
\begin{corollary}\label{cor2D}
Let $X\subseteq \R^2$ be a compact $r$-regular set observed on a stationary non-isotropic lattice and let $\xi_l$ be a con\-fi\-gu\-ra\-tion. Then
\begin{align*}
\lim_{a\to 0}{}&(EN_l-a^{-1}\lim_{a\to 0}aEN_l)\\ 
&=\frac{1}{2} \int_{\partial X}(2(h(B_l,n)^2-h(W_l,n)^2)-(|p_{B_l}^+|^2-|p_{W_l}^-|^2))\delta_{(B_l,W_l)}dC_0(X;\cdot)\\
&=\frac{1}{2\pi}\bar{\mu}_lV_0(X).
\end{align*}
Here $C_0(X;\cdot)$ is the 0th curvature measure given by $C_0(X;A)=\int_{A\cap \partial X} k d\Ha^1$. 
\end{corollary}
The second equality uses the identity $C_0(X;\cdot) \circ n^{-1}=2\pi V_0(X)\Ha^1$ as measures on $S^1$. 

From this we first obtain the following criterion for the existence of an asymptotic mean:
\begin{proposition}\label{2Dconv}
$\lim_{a\to 0}E\hat{V}_0(X)$ exists for all $X$ if and only if 
\begin{equation}\label{wcondition}
w_2^{(0)}=0 \textrm{ and } w_1^{(0)}=-w_3^{(0)}.
\end{equation}
\end{proposition}

\begin{proof}
By Corollary \ref{cor2D}, $\lim_{a\to 0}E\hat{V}_0(X)$ exists if and only if
\begin{equation}\label{limval}
\sum_{j=1}^3 w_j^{(0)} \bar{\varphi}_j(X)=0.
\end{equation}
Write $n=(n_1,n_2)\in S^1\subseteq \R^2$. Then for $j=1,3$,
\begin{equation*}
\sum_{l:\xi_l\in \eta_j^2}(-h(B_l\oplus\check{W}_{l},n))^+=\min\{|n_1|,|n_2|\},
\end{equation*}
wheras
\begin{equation*}
\sum_{l:\xi_l\in \eta_2^2} (-h(B_l\oplus\check{W}_{l},n))^+=\max\{|n_1|,|n_2|\}-\min\{|n_1|,|n_2|\}.
\end{equation*}
Hence the equation \eqref{limval} becomes
\begin{equation*}
\int_{\partial X}\big(\big(w_1^{(0)}+w_3^{(0)}-w_2^{(0)}\big)\min\{|n_1|,|n_2|\}+w^{(0)}_2\max\{|n_1|,|n_2|\}\big)d\Ha^{1}=0.
\end{equation*}
This holds for all $X$ if $w_1^{(0)}+w_3^{(0)}=w_2^{(0)}=0$. On the other hand, this is a necessary condition, as one may realize e.g.\ by considering sets of the form $[0,(0,x)]\oplus B(r)$ where $[x,y] $ denotes the line segment from $x$ to $y$.\qed
\end{proof}

\begin{theorem}\label{2Dunbiased}
For an estimator satisfying \eqref{wcondition},
\begin{equation*}
\lim_{a\to 0}E\hat{V}_0(X)= 2\big(w_1^{(0)}-w_3^{(0)}\big)V_0(X).
\end{equation*}
Thus the estimator with weights
\begin{equation*}
w_1^{(0)}=-w_3^{(0)}=\frac{1}{4} \textrm{ and } w_2^{(0)}=0
\end{equation*}
is the unique asymptotically unbiased estimator for the Euler characteristic of the form \eqref{motion} in the non-isotropic setting.
\end{theorem}

\begin{proof}
Under the condition \eqref{wcondition}, $\lim_{a\to 0} E\hat{V}_0(X)$ is given by Corollary \ref{cor2D} if we can compute the coefficients $\bar{\mu}_j$.
This is done in \cite[Section 8]{am} and it yields
\begin{equation*}
\lim_{a\to 0} E\hat{V}_0(X) = 2\big(w_1^{(0)}-w_3^{(0)}\big)V_0(X)=4w_1^{(0)}V_0(X)
\end{equation*}
as claimed.\qed
\end{proof}

\section{Non-existence of unbiased estimators for $V_{d-2}$ in higher dimensions}\label{higher}
We now consider estimators of the form \eqref{motion} for $V_{d-2}$ in dimensions $d\geq 3$ in the design based setting where an $r$-regular set $X\subseteq \R^d$ is observed on a stationary non-isotropic lattice $a\La$.
Contrary to the $d=2$ case, we shall see that in higher dimensions there are no asymptotically unbiased estimators based on $2\times \dotsm \times 2$ configurations. The proof goes by constructing counterexamples. These are all of the form $P\oplus B(r)$ where $P$ is a polygon.

We first show a small lemma that will simplify the proofs:
\begin{lemma} \label{kryds}
Let $\xi_l$ be a configuration. For $u_1,\dots,u_k\in \R^d\backslash \{0\}$ or\-tho\-go\-nal and $X = (\bigoplus_{i=1}^k [0,u_i]) \times S^{d-k-1}(u_1,\dots , u_k)$,
\begin{equation*}
\int_{ X}(\II^+(B_l)-\II^-(W_l))^+\mathds{1}_{\{h(B_l\oplus \check{W}_l,n)=0\}}d\Ha^{d-1} =0.
\end{equation*}
\end{lemma}

Here $S^{d-k-1}(u_1,\dots , u_k)$ denotes the unit sphere in $\spa(u_1,\dots,u_k)^\perp $.
\begin{proof}
If $h(B_l\oplus \check{W}_l,n)=0$, there are $b\in B_l$ and $w \in W_l$ with $\II^+(B_l)=\II(b)$, $\II^-(W_l)=\II(w)$, and $\langle b-w,n\rangle$. Let $v=b-w \neq 0$ and for $y\in \R^d$, write $y=y_1+y_2$ where $y_1$ is the projection of $y$ onto
$\spa(u_1,\dots,u_k)$. Observe that $n(x) = n_2(x) $ for all $x\in X$. Thus the set $\{x\in X \mid \langle n,v \rangle=\langle n_2,v_2 \rangle=0\}$ can only have positive $\Ha^{d-1}$-measure if $v_2=0$, that is, if $b_2=w_2$. But then the claim follows since $\II(b)=\II(b_2)=\II(w_2)=\II(w)$.
\end{proof}

\begin{theorem}\label{non}
For $d=3 $, there exists no asymptotically unbiased estimator for $V_{1}$ of the form \eqref{motion} on the class of $r$-regular sets.
\end{theorem}

In the following we write $w_j=w_j^{(d-2)}$ for simplicity.

\begin{proof}
Assume that $\hat{V}_1$ is an estimator of the form \eqref{motion} and  that the weights have been chosen so that $\lim_{a\to 0}aE\hat{V}_{1}(X) = 0$ and $\lim_{a\to 0}E\hat{V}_{1}(X) = V_1(X)$ for all $r$-regular sets $X$.

In particular, this holds for $X=B(r)$.  Since  $X$ is rotation invariant, a random rotation of $\La$ does not change $EN_l$. Thus $\bar{\lambda}_l(X)=\bar{\mu}_lV_{d-2}(B(r))$, so it follows from Theorem \ref{3Dlim} that the weights must satisfy
\begin{equation}\label{wer}
(3-\sqrt{3})(w_1-w_7)+(3\sqrt{3}-3\sqrt{2})(w_2-w_6)+(-3+6\sqrt{2}-3\sqrt{3})(w_3-w_5)=1.
\end{equation}


We next consider three test sets of the form $X_i=[0,t_iu_i]\oplus B(r)$ for $t_i\in \R$ and $u_1=(1,0,0)$, $u_2=\big(\tfrac{1}{\sqrt{2}},\tfrac{1}{\sqrt{2}},0\big)$ and $u_3=\big(\tfrac{1}{\sqrt{3}},\tfrac{1}{\sqrt{3}},\tfrac{1}{\sqrt{3}}\big)$. Then 
\begin{equation}\label{V1val}
V_1(X_i)= t_i+ 4r =t_i +V_1(B(r)).
\end{equation}

Note that 
\begin{equation*}
\partial X_i = (0+rS^2\cap H^-_{u_i})\cup (t_iu_i+rS^2\cap H^+_{u_i})\cup ([0,t_iu_i]\times rS^1(u_i))
\end{equation*}
 where $H^{\pm}_{u_i}$ denote the halfspaces $\{z\in \R^3\mid \pm \langle z , u_i \rangle \geq 0 \}$ and $rS^1(u_i)$ is the  sphere of radius $r$ in $u_i^\perp$. 
%
Thus by Lemma \ref{kryds},
\begin{align*}
\lambda_l(X)={}& \frac{1}{2}\int_{[0,t_iu_i]\times rS^1(u_i) }(Q^+(B_l)-Q^-(\check{W}_l))\delta_{(B_l,W_l)}d\Ha^{d-1}\\
&+ \frac{1}{2}\int_{rS^{2}}(Q^+(B_l)-Q^-({W}_l))\delta_{(B_l,W_l)}d\Ha^{d-1} \\
={}& \frac{1}{2}\int_{[0,t_iu_i]\times rS^1(u_i) }(Q^+(B_l)-Q^-({W}_l))\delta_{(B_l,W_l)}d\Ha^{d-1} +\lambda_l(B(r)).
\end{align*} 
Combining this with Corollary \ref{confcor} yields
\begin{align*}
\lim_{a\to 0} E{}&\hat{V}_{1}(X_i)-\lim_{a\to 0} E\hat{V}_{1}(B(r))\\
={}&\sum_{j\in J}w_j \sum_{l:\xi_l\in \eta_j^3}\frac{1}{2}\int_{[0,t_iu_i]\times rS^1(u_i) }(Q^+(B_l)-Q^-({W}_l))\delta_{(B_l,W_l)}d\Ha^{d-1}.
\end{align*}

Under the assumption that $\hat{V}_1$ is asymptotically unbiased on both $B(r)$ and $X_i$, \eqref{V1val} shows that the weights must satisfy 
\begin{equation*}
h_i:=\sum_{j\in J} w_j \sum_{l:\xi_l\in \eta_j^3}\frac{1}{2}\int_{[0,t_iu_i]\times rS^1(u_i) }(Q^+(B_l)-Q^-({W}_l))\delta_{(B_l,W_l)}d\Ha^{d-1} = t_i
\end{equation*}
for $i=1,2,3$.

But $Q$ takes a very simple form on $[0,t_iu_i]\times rS^1(u_i)$. Namely, for $t\in [0,t_i]$ and $n\in S^1(u_i)$,
\begin{equation*}
Q_{tu_i+rn}(s)=\tfrac{1}{r}(\langle s,n\rangle^2-\langle s,u_i\times n \rangle^2)
\end{equation*}
where $\times $ is the cross-product in $\R^3$. In particular, $Q_{tu_i+rn}(s)$ depends only on $n$ and the projection of $s$ onto $u_i^\perp$. Hence 
\begin{align*}
h_i=t_i\sum_{j\in J} w_j \sum_{l:\xi_l \in \eta_j^3} \frac{1}{2}\int_{ S^1(u_i)}{}&\Big( \langle b_l^+,n\rangle^2-\langle b_l^+,u_i\times n \rangle^2  \\
& -\langle w_l^-,n\rangle^2+\langle w_l^-,u_i\times n \rangle^2 \Big)\delta_{(B_l,W_l)}(n)\Ha^{1}(dn).
\end{align*}
It is now a straightforward computation to see that 
\begin{align*}
h_1&=2(w_2-w_6)t_1,\\
h_2&=(\sqrt{2}(w_1-w_7)+\sqrt{2}(w_3-w_5))t_2,\\
h_3&=(\sqrt{3}(w_1-w_7)+\sqrt{3}(w_2-w_6)-\sqrt{3}(w_3-w_5))t_3.
\end{align*} 
But no weights can satisfy the three equations $h_i=t_i$ and Equation \eqref{wer} at the same time.\qed
\end{proof}

\begin{theorem}\label{main'}
There are no asymptotically unbiased estimators for $V_{d-2}$ of the form \eqref{motion} in dimension $d\geq 3$.
\end{theorem}

For shortness we write 
\begin{equation*}
G_j=\frac{1}{2}\sum_{l:\xi_l\in \eta_j}(Q^+(B_l)-Q^-({W}_l))\delta_{(B_l,W_l)}
\end{equation*}
in the following.

\begin{proof}
The idea is to generalize the approach for $d=3$ by considering some example sets for which the computations reduce to the ones already performed in dimension~3. Again we assume that an asymptotically unbiased estimator $\hat{V}_{d-2}$ is given.

Let $u_1,\dots ,u_k\in S^{d-1}$ be $k\leq d-2$ orthonormal vectors. We consider sets of the form  
\begin{equation*}
([0,t_1u_1]\oplus \dotsm \oplus [0,t_ku_k])\times rS^{d-k-1}(u_1,\dots,u_k)
\end{equation*}
 where $t_i> 0$.

We first show by induction in $k$ that the weights must satisfy
\begin{equation}\label{square}
\sum_{j\in J} w_j \int_{(\bigoplus_{i=1}^k[0,t_iu_i] )\times rS^{d-k-1}(u_1,\dots,u_k)}G_jd\Ha^{d-1}=\frac{\kappa_{d-k}}{\kappa_2}\binom{d-k}{2}r^{d-k-2}\prod_{i=1}^k t_i
\end{equation}
where $\kappa_N$ is the volume of the unit ball in $\R^N$.
This is obviously true for $k=0$ since the estimator is unbiased for $X=B(r)$. Assume it is true for $k-1$ and consider $ X= P\oplus B(r)$ where $P=\bigoplus_{i=1}^k [0,t_iu_i]$.
The relative open $m$-faces of $P$ are the sets
\begin{equation*}
x+\bigoplus_{i=1}^m (0,t_{k_i}u_{k_i}) 
\end{equation*}
for
\begin{equation*}
x\in A({k_1},\dots,{k_m})=\bigg\{ \sum_{s\neq k_1,\dots k_m} \eps_s t_s u_s \mid \eps_s\in\{0,1\}\bigg\}.
\end{equation*}
The normal cone of such a face is
\begin{equation*}
N(x,k_1,\dots k_m)=\bigcap_{s\neq k_1,\dots k_m} H^+_{(-1)^{\eps_s-1}u_s}\cap \text{span}(u_{k_1},\dots,u_{k_m})^\perp.
\end{equation*}
 Then $\partial X$ can be divided into disjoint subsets of the form 
\begin{equation*}
x+\bigg(\bigoplus_{i=1}^m (0,t_{k_i}u_{k_i})\bigg) \times (N(x,k_1,\dots ,k_m)\cap rS^{d-1})
\end{equation*}
for $x\in A({k_1},\dots,{k_m})$.
Note that
\begin{equation}\label{sphere}
\bigcup_{x\in A({k_1},\dots,{k_m})}N(x,k_1,\dots k_m)\cap rS^{d-1} = rS^{d-m-1} (u_{k_1},\dots,u_{k_m})
\end{equation}
and for $x_1\neq x_2$, 
\begin{equation*}
N(x_1,k_1,\dots k_m)\cap N(x_2,k_1,\dots k_m)\cap rS^{d-1} 
\end{equation*}
has $\Ha^{d-m-1}$-measure zero in $rS^{d-m-1} (u_{k_1},\dots,u_{k_m})$.
Thus for $m< k$,
\begin{align*}
\sum_{j\in J} w_j{}& \sum_{x\in A({k_1},\dots,{k_m})}\int_{x+(\bigoplus_{i=1}^m (0,t_{k_i}u_{k_i}))\times (N(x,k_1,\dots,k_m)\cap rS^{d-1})}G_jd\Ha^{d-1}\\
&=\sum_{j\in J} w_j \int_{(\bigoplus_{i=1}^m(0,t_{k_i}u_{k_i}))\times rS^{d-m-1}(u_{k_1},\dots,u_{k_m})}G_jd\Ha^{d-1}\\
&= \frac{\kappa_{d-m}}{\kappa_2}\binom{d-m}{2}r^{d-m-2}\prod_{i=1}^m t_{k_i} 
\end{align*}
where the last equality follows by induction. 
But then it must hold for $m=k$ as well since on the one hand $\lim_{a\to 0}E\hat{V}_{d-2}(P\oplus B(r))$ equals
\begin{equation*}
\sum_{j\in J} w_j\sum_{m=0}^k\, \sum_{\substack{1\leq k_1<\dotsm < k_m \leq k,\\ x\in A({k_1},\dots,{k_m})}} \, \int_{x+(\bigoplus_{i=1}^m (0,t_{k_i}u_{k_i}))\times (N(x,k_1,\dots,k_m)\cap rS^{d-1})}G_jd\Ha^{d-1}
\end{equation*}
by Lemma \ref{kryds}, while on the other hand, the Steiner formula yields
\begin{align*}
V_{d-2}\left(P\oplus B(r)\right)&=\frac{1}{\kappa_2}\sum_{m=0}^{d-2} \binom{d-m}{2}r^{d-m-2}\kappa_{d-m}V_m(P) \\
&= \frac{1}{\kappa_2}\sum_{m=0}^{d-2} \binom{d-m}{2}r^{d-m-2}\kappa_{d-m}\sum_{1\leq k_1<\dotsm < k_m\leq k}\prod_{i=1}^m t_{k_i}.
\end{align*}
Here the last equality uses \cite[Equation (4.2.30)]{schneider} and the observation \eqref{sphere}.
This proves the induction step. 

In particular, \eqref{square} must hold for $k=d-2$ and the orthonormal vectors $u_i,e_4,\dots,e_d$ where $u_i\in \text{span}(e_1,e_2,e_3)$ are defined as in Theorem \ref{non} for $i=1,2,3$. That is,
\begin{equation}\label{tieq}
\sum_{j\in J} w_j \int_{\left([0,t_iu_i]\oplus \bigoplus_{m=4}^d[0,e_m]\right)\times rS^{1}(u_i,e_4,\dots,e_d)}G_jd\Ha^{d-1}= t_i.
\end{equation} 

If $\xi_l\subseteq \text{span}(e_1,e_2,e_3)\cong \R^3$ is a configuration in $\R^3$, we let $\xi_l'\subseteq \R^d$ denote the configuration $C_0 \cap P^{-1}(\xi_l)$ where $P: \R^d \to \text{span}(e_1,e_2,e_3)$ is the projection.
If $\xi_{l_1}$ and $\xi_{l_2}$ differ only by a rigid motion, so do $\xi_{l_1}'$ and $\xi_{l_2}'$. If the configuration classes $\eta_j^3$ in $\R^3$ are indexed by $j\in J$ and $\xi_l\in \eta_j^3$, we let $\eta_j^d$, $j\in J$, denote the configuration class of $\xi_l'$.

For  $x\in ([0,t_iu_i]\oplus \bigoplus_{m=4}^d[0,e_m])\times rS^{1}(u_i,e_4,\dots,e_d)$,
\begin{equation*}
\delta_{(B_l,W_l)}(n(x))=\delta_{(PB_l,PW_l)}(n(x)).
\end{equation*}
Thus only configurations of type $\eta_j^d$ with $j\in J$ can occur. Moreover, since all principal curvatures vanish in the directions $u_i,e_4,\dots,e_d$, 
\begin{align*}
\sum_{j\in J} {}&w_j \int_{\left([0,t_iu_i]\oplus \bigoplus_{m=4}^d [0,e_m]\right)\times rS^{1}(u_i,e_4,\dots,e_d)}G_jd\Ha^{d-1}\\
&=\sum_{j\in J} w_j \sum_{l:\xi_l\in \eta_j^d}\frac{1}{2}\int_{[0,t_iu_i] \times rS^{1}(u_i)}(Q^+(PB_l)-Q^-(P{W_l}))\delta_{(PB_l,PW_l)}d\Ha^{d-1}\\
&=h_i.
\end{align*}
where $h_i$ is as in the proof of Theorem \ref{non}.
Thus by \eqref{tieq} the weights must satisfy the equations $h_i=t_i$. 

Applying \eqref{square} to the $k=d-3$ vectors $e_4,\dots,e_d$ shows that the weights must also satisfy \eqref{wer}. But then the $w_j$ have to satisfy the same set of equations as in the proof of Theorem \ref{non}, which was impossible.\qed
\end{proof}

\begin{acknowledgements}
The author is supported by the Centre for Stochastic Geometry and Advanced Bioimaging, funded by the Villum Foundation.
The author is most thankful to Markus Kiderlen for helpful suggestions and proofreading. 
\end{acknowledgements}

\end{document}